\documentclass[11 pt]{article}
\usepackage{amsmath}
\usepackage{amsthm}
\usepackage{graphicx}
\usepackage{bbm,amssymb}
\usepackage[hyperfootnotes=false,colorlinks]{hyperref}

\newcommand{\floor}[1]{\lfloor {#1} \rfloor}

\newcommand{\old}[1]{}
\newcommand{\arxiv}[1]{{\tt \href{http://arxiv.org/abs/#1}{arXiv:#1}}}

\newtheorem{theorem}{Theorem}

\newtheorem{lemma}[theorem]{Lemma}

\newtheorem*{lemmaA}{Lemma A}

\theoremstyle{remark}
\newtheorem*{remark}{Remark}

\theoremstyle{definition}
\newtheorem*{definition}{Definition}

\def\Grid{\mathcal{G}}
\def\Early{\mathcal{E}}
\def\Late{\mathcal{L}}
\def\Re{\mathrm{Re}}

\def\one{\mathbf{1}}
\def\B{\mathbf{B}} 
\def\aa{\mathbf{a}}

\title{Logarithmic Fluctuations for Internal DLA}
\author{David Jerison\footnote{Partially supported by NSF grant DMS-1069225} \and Lionel Levine\footnote{Supported by an NSF Postdoctoral Research Fellowship.} \and Scott Sheffield\footnote{Partially supported by NSF grant
DMS-0645585.}}
\date{July 6, 2011}

\DeclareSymbolFont{AMSb}{U}{msb}{m}{n}
\DeclareMathSymbol{\C}{\mathbin}{AMSb}{"43}
\DeclareMathSymbol{\EE}{\mathbin}{AMSb}{"45}
\DeclareMathSymbol{\N}{\mathbin}{AMSb}{"4E}
\DeclareMathSymbol{\PP}{\mathbin}{AMSb}{"50}
\DeclareMathSymbol{\Q}{\mathbin}{AMSb}{"51}
\DeclareMathSymbol{\R}{\mathbin}{AMSb}{"52}
\DeclareMathSymbol{\Z}{\mathbin}{AMSb}{"5A}

\begin{document}

\maketitle
\renewcommand{\thefootnote}{}
\renewcommand{\thefootnote}{\arabic{footnote}}

\begin{abstract}
Let each of $n$ particles starting at the origin in $\mathbb Z^2$ perform simple random walk until reaching a site with no other particles.  Lawler, Bramson, and Griffeath proved that the resulting random set $A(n)$ of $n$ occupied sites is (with high probability) close to a disk $\B_r$ of radius $r=\sqrt{n/\pi}$.
We show that the discrepancy between $A(n)$ and the disk is at most logarithmic in the radius: i.e., there is an absolute constant~$C$ such that with probability $1$,
	\[ \B_{r - C\log r} \subset A(\pi r^2) \subset \B_{r+ C\log r} \quad \mbox{ for all sufficiently large $r$}. \]
\end{abstract}

\pagebreak
\tableofcontents

\section{Introduction}

In the process known as internal diffusion limited aggregation (IDLA)
one constructs for each integer
time $n \geq 0$ an {\bf occupied set} $A(n) \subset \mathbb Z^2$ as follows:
begin with $A(1) = \{0\}$, and for each $n$
let $A(n+1)$ be $A(n)$ plus the first point at which a simple random walk from the origin
hits $\Z^2 \setminus A(n)$.

IDLA was first proposed by Meakin and Deutch \cite{MD} as a model of industrial chemical processes such as electropolishing, corrosion and etching.  After discussing these applications, they write ``it is also of some fundamental significance to know just how smooth a surface formed by diffusion limited processes may be.''  Their numerical findings indicated that such surfaces are astonishingly smooth: in two and higher dimensions, the fluctuations appeared to be at most logarithmic in the number of particles.

In the more than two decades since, this finding has resisted mathematical proof.  Diaconis and Fulton \cite{DF} identified IDLA as a special case of a ``smash sum'' operation on subsets of $\Z^2$.  Lawler, Bramson and Griffeath~\cite{LBG} proved that the asymptotic shape of the domain $A(n)$ is a disk, and Lawler~\cite{Lawler95} showed that the fluctuations from circularity are at most of order $r^{1/3}$ up to logarithmic factors, where $r=\sqrt{n/\pi}$ is the radius of the disk of area~$n$.   Later work \cite{GQ,LP10} related internal DLA with multiple sources to variational problems in PDE.  However, the basic issue raised by Meakin and Deutch and again by Lawler (``A more interesting question... is whether the errors are $o(n^\alpha)$ for some $\alpha<1/3$.'' \cite{Lawler95})  stood until Asselah and Gaudilli\`{e}re~\cite{AG10a} (with an appendix by Blach\`{e}re) improved the bound of $r^{1/3}$ in dimensions $d\geq 3$ to $r^{1/(d+1)}$.

Our main result shows that in dimension $2$, the fluctuations from circularity are at most order $\log r$.  In fact we show somewhat more, namely that the probability that the fluctuations at time $n=\pi r^2$ exceed $a \log r$ decays like $r^{-\gamma}$ (where $\gamma$ can taken to be arbitrarily large, but the constant $a$ may depend on $\gamma$).  Let
	\[ \B_r = \left\{ z \in \Z^2 \mid z_1^2 + z_2^2 < r^2 \right\} \]
be the set of lattice points lying in the open disk of radius $r$ centered at~$0$.  For convenience, we define the IDLA cluster $A(t)$ for all $t \in \R$ by setting $A(t) = A(\floor{t})$, where $\floor{t}$ is the greatest integer $\leq t$ (and $A(t)=\emptyset$ if $t < 1$).

\begin{theorem}\label{thm:logfluctuations}
For each $\gamma$ there exists an $a = a(\gamma) < \infty$
such that for all sufficiently large $r$,
\[
\PP \Big( \left\{ \B_{r - a \log r} \subset A(\pi r^2) \subset \B_{r+ a\log r} \right\}^c \Big)
\le r^{-\gamma}. \]
\end{theorem}

The proof will show that $a(\gamma)$ can be taken to be an affine function of~$\gamma$.  Taking $C=a(\gamma)$ for large enough $\gamma$, Theorem~\ref{thm:logfluctuations} in particular implies (by Borel-Cantelli) that the fluctuations from circularity are $O(\log r)$ almost surely:
	\begin{equation} \label{e.logfluctuation} \PP \left\{ \B_{r - C\log r} \subset A(\pi r^2) \subset \B_{r+ C\log r} \; \mbox{ for all sufficiently large~$r$} \right\} = 1. \end{equation}
	
This is the first part of a three part article.  In the second part, entitled \emph{Internal DLA in higher dimensions} \cite{JLS10}, we show that internal DLA in $\Z^d$ for $d \geq 3$ has even smaller fluctuations than in dimension~$2$: the higher-dimensional analogue of \eqref{e.logfluctuation} holds with $\log r$ replaced by $\sqrt{\log r}$.  We believe that these orders ($\log r$ in dimension $2$ and $\sqrt{\log r}$ in dimensions $d\geq 3$) are the best possible.  We discuss a possible approach to proving matching lower bounds, as well as several other open questions, in Section~\ref{s.conclusion}. Using methods rather different from ours, Asselah and Gaudilli\`{e}re~\cite{AG10b} have recently announced an independently obtained upper bound of order $\log^2 r$ in all dimensions $d\geq 2$.

In the third part, entitled \emph{Internal DLA and the Gaussian free field} \cite{JLS11}, we show that the fluctuations from circularity have a weak limit which is a variant of the Gaussian free field.  While both this paper and the sequel address fluctuations of IDLA, the results are of a rather different nature: the present paper is concerned with fluctuations at the fine scale of individual lattice points.  The sequel addresses average fluctuations, and will show that in a certain sense the fluctuation from circularity {\em averaged} over a constant fraction of the boundary (here inner and outer fluctuations cancel each other out) is only {\em constant} order.  In fact, the fluctuations from circularity converge in law to a particular Gaussian random distribution on the circle, whose time evolution is related to the Gaussian free field.

\begin{figure}[htbp]
\begin{center}
\includegraphics[height=.5\textheight]{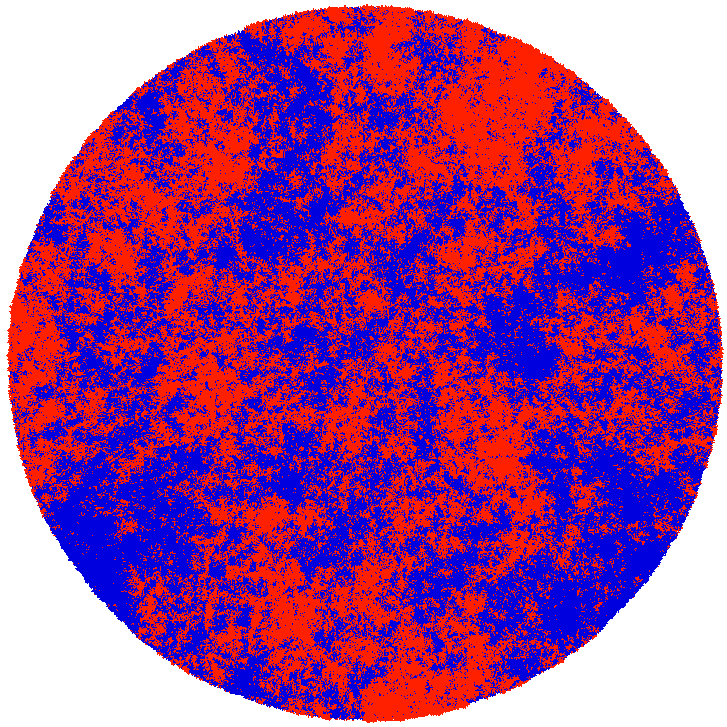} \\ \vspace{3mm}
\includegraphics[width=\textwidth]{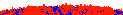}
\caption{\label{fig:cluster} Top: IDLA cluster $A(n)$ for $n=4\times 10^5$.
Early points (see Section~\ref{sec:basecase}) are colored red, and late points are colored blue.  Bottom: closeup of a portion of the boundary of the cluster.}
\end{center}
\end{figure}

\section{Logarithmic fluctuations} \label{s.logsection}

We begin in Section~\ref{sec:nothintentacles} with a lemma that rules out ``thin tentacles'' in the IDLA cluster (Lemma~\ref{lem:thickening}).  More precisely, it is unlikely that $z \in A(n)$ and less than a constant fraction $b m^2$ of lattice sites in the disk $\B(z,m)$ belong to $A(n)$.  Section~\ref{sec:basecase} establishes a rather rough a priori bound on the probability of certain types of very large fluctuations.  Section~\ref{sec:exittimes} gives some large deviation results for the exit time of Brownian motion from an interval.  These three sections comprise the background results needed for our argument.

In order to control more precisely the deviation from circularity near a point $\zeta \in \Z^2$, we define in Section~\ref{sec:hzetaproperties} a discrete harmonic function $H_\zeta: \Z^2 \to \R$ which approximates the harmonic function \begin{equation} \label{e.fzeta} F_\zeta(z) := \Re \left( \frac{\zeta / |\zeta|}{\zeta-z} \right), \end{equation} where $z$ and $\zeta$ are viewed as complex variables.  (Discrete harmonicity will fail for $H_\zeta$ only at a few points near $\zeta$.)  The function $H_\zeta(z) - \frac{1}{2|\zeta|}$ approximates the discrete Poisson kernel for the disk $\B_{|\zeta|}$, that is, the probability that simple random walk started at $z$ first exits $\B_{|\zeta|}$ at $\zeta$.
 We may extend the function $H_\zeta$ linearly to the grid~$\Grid$ of horizontal and vertical line segments joining vertices of $\Z^2$.  If we begin with a collection of particles on the grid~$\Grid$, and we allow some of them to move according to Brownian motions on~$\Grid$ (see Section~\ref{sec:martingale}), according to determined starting and stopping rules, then the sum of~$H_\zeta$ over these particle locations is a continuous martingale (provided that no particle continues to move after reaching a point where discrete harmonicity fails).

\begin {figure}[htbp]
\begin {center}
\includegraphics [width=2in]{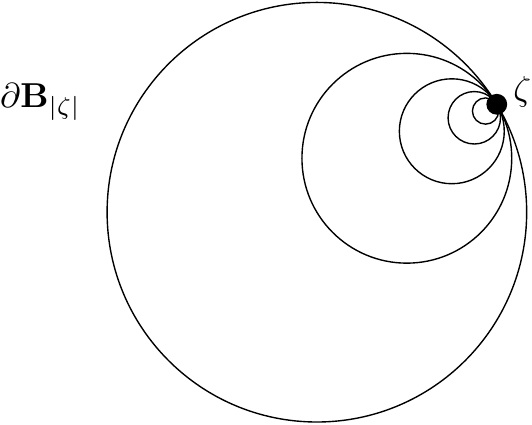}
\caption {\label{Hlevels} Level sets of $F_\zeta$.  Note that $F_\zeta=\frac{1}{2|\zeta|}$ on $\partial \B_{|\zeta|}$ and $F_\zeta(0) = \frac{1}{|\zeta|}$.}
\end {center}
\end {figure}

By the martingale representation theorem (see~\cite[Theorem~V.1.6]{RY}), we can represent any continuous martingale as a time change of a standard Brownian motion.
Using the representation theorem and large deviations for Brownian exit times, we can control the likelihood that this martingale is larger or smaller than its mean by estimating its quadratic variation.

We will ultimately show that large deviations of $A(n)$ from circularity are unlikely without large deviations of the martingale for some $\zeta$.  For example, the fact that thin tentacles are unlikely (Lemma~\ref{lem:thickening}) means that with high probability, a point near some~$\zeta$ cannot belong to $A(n)$ unless many points near~$\zeta$ belong to $A(n)$.  Since $H_\zeta(z)$ is large for~$z$ close to~$\zeta$, this can be used to show that (when $\zeta$ is outside the typical range of $A(n)$) the martingale is likely to be large if a point near $\zeta$ belongs to $A(n)$.  Similarly, we can show that if $\zeta$ fails to be hit (long after we expect it to have been hit) then the martingale has to be small at that time.  (For this step we actually use a modified version of the process in which particles are frozen when they first exit a set called $\Omega_{\zeta}$ that approximates $\B_{|\zeta|}$, see Figure \ref{fig:hzeta}.)

The martingale is a Brownian motion parameterized by its quadratic variation time, but the amount of quadratic variation time elapsed while $A(n)$ is constructed is random.  Note that the conditional expected amount of time elapsed when depositing one particle (given the locations of all previous particles) is the conditional variance of the value of $H_\zeta$ at the position that the particle is deposited.  This value is easy to approximate on the event that $A(n)$ is approximately circular.  If, as a heuristic, we {\em assume} that $A(n)$ is approximately circular for all $n$, then a back of the envelope calculation shows that the total time lapse in constructing $A(\pi |\zeta|^2)$ should be of order $\log |\zeta|$ in dimension two (and of constant order in all higher dimensions).  On the event that the total time lapse is about right, the probability of having a martingale fluctuation of $C \sqrt{\log |\zeta|}$ standard deviations should decay like $e^{-\frac{C^2}{2} \log |\zeta|} = |\zeta|^{-C^2/2}$.  If $C$ is large enough then we should be able to sum over all $\zeta$ and show that we are unlikely to see fluctuations of this magnitude anywhere.

This heuristic actually suggests the basic strategy of our argument: we assume some amount of regularity for the growth of $A(n)$ and use that assumption to show that the quadratic variation time lapse is unlikely to be large, hence large martingale fluctuations are unlikely, and hence the growth of $A(n)$ is {\em even more regular} than we assumed.  This strategy is pursued in Sections~\ref{sec:earlyimplieslate} and~\ref{sec:lateimpliesearly}, which contain the heart of proof.  Using the ideas above, we will show that the presence of an ``early point'' (a point $z$ visited at time less than $\pi |z|^2$) implies, with high probability, the presence of a comparably late point at a previous time (Lemma~\ref{lem:earlyimplieslate}); similarly, the presence of a late point implies, with high probability, the presence of a {\em significantly earlier} point at a previous time (Lemma~\ref{lem:lateimpliesearly}).  By iterating these two lemmas, we will see that if there were a substantial probability to have even a slightly early or late point (off by more than $C \log r$), then there would be a substantial probability of having a much larger fluctuation, which would contradict our a priori estimate.

In general, we use $C_0, C_1, \ldots$ to denote large constants and $c_0, c_1, \ldots$ to denote small constants.  Unless explicitly specified otherwise, these are all positive absolute constants.

\subsection{No thin tentacles}
\label{sec:nothintentacles}

Let $A(n)$ be the internal DLA cluster formed from $n$ particles started at the origin in $\Z^2$.  For $z \in \Z^2$ and a positive real number~$m$, let
	\[ \B(z,m) = \{y \in \Z^2 \,:\, |y-z| < m \} \]
where $|z| = (z_1^2+z_2^2)^{1/2}$ is the Euclidean norm.

The following lemma is a variant of \cite[Lemma 5.12]{LP10}.  The basic estimates involved are similar to those used by Lawler, Bramson and Griffeath in their proof of the outer bound; see in particular \cite[Lemma~5b]{LBG}.

\begin{lemma}
\label{lem:thickening}
There are positive absolute constants $b$, $C_0$, and $c_0$ such that
for all real numbers $m >0$ and all $z \in \Z^2$ with $|z| \ge m$,
\begin{equation}
\label{e.thintentaclebound}
\PP \left\{ z \in A(n),~ \# (A(n) \cap \B(z,m)) \leq b m^2 \right\} \leq
C_0 e^{-c_0 m}.
\end{equation}
\end{lemma}

In other words, it is unlikely that $z$ belongs to the cluster $A(n)$ but less than a constant fraction of sites in the ball $\B(z,m)$ belong to the cluster.  Figure~\ref{fig:tentacle} depicts this unlikely scenario.

\begin {figure}[htbp]
\begin {center}
\includegraphics [width=2in]{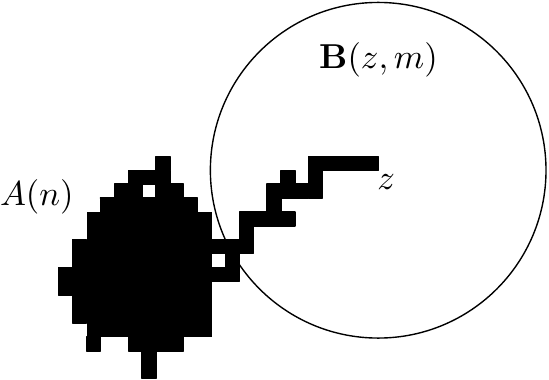}
\caption {\label{fig:tentacle} A thin tentacle.}
\end {center}
\end {figure}

In fact, the exponent on the right side of \eqref{e.thintentaclebound} can be improved to order $m^2/ \log m$ (and to order $m^2$ in dimensions $d \geq 3$).  In appendix~\ref{sec:tentacle}, we prove these stronger bounds in all dimensions.  We anticipate using these bounds for $d\geq 3$ in our sequel paper \emph{Internal DLA in higher dimensions}.  For purposes of the present paper, we will use only the bound \eqref{e.thintentaclebound} in dimension $d=2$.

\subsection{A priori bound on probability of large fluctuations}
\label{sec:basecase}

We expect a point $z \in \Z^2$ to first join the IDLA cluster $A(n)$ at time about $n=\pi |z|^2$.  This motivates the following two definitions, which are illustrated in Figure~\ref{fig:earlylate}.

\begin{definition}\label{def:early}
For a positive real number $m$, we say that $z\in \Z^2$ is \textbf{$m$-early} if
$z \in A(\pi(|z|-m)^2)$.  For a time $N$, let
	\[ \Early_m[N] := \bigcup_{z \in A(N)} \left\{z \in A(\pi(|z|-m)^2) \right\} \]
be the event that some point in $A(N)$ is $m$-early.
\end{definition}

\begin{definition}\label{def:late}
For a positive real number $\ell$, we say that $z\in \Z^2$ is \textbf{$\ell$-late} if $z \notin A(\pi (|z|+\ell)^2)$.  For a time $N$, let
	\[ \Late_\ell[N] = \bigcup_{z \in \B_{\sqrt{N/\pi}-\ell}} \left\{ z \notin A(\pi (|z|+\ell)^2) \right\} \]
be the event that some point in $\B_{\sqrt{N/\pi}-\ell}$ is $\ell$-late.
\end{definition}

\begin {figure}[htbp]
\begin {center}
\includegraphics [width=5in]{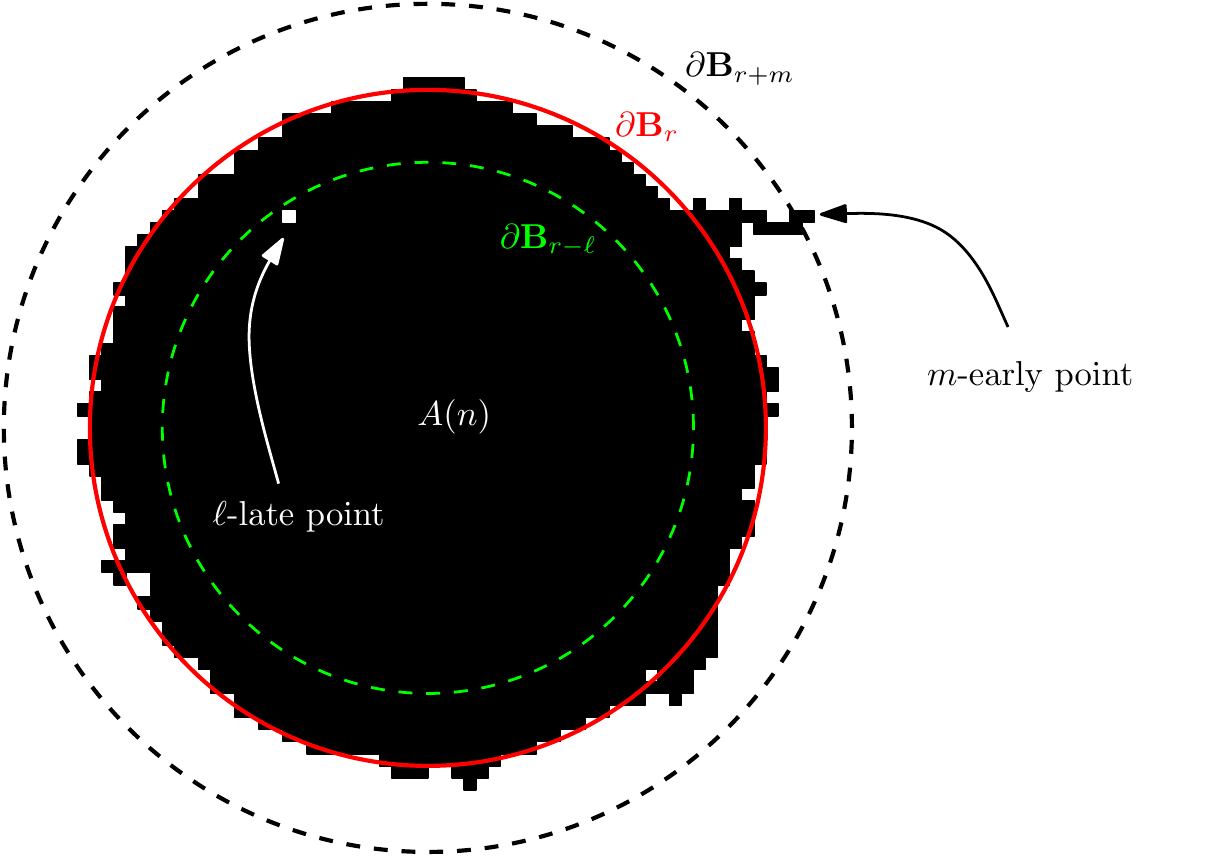}
\caption {\label{fig:earlylate} Illustration of an IDLA cluster $A(n)$ (union of the black squares), an $m$-early point and an $\ell$-late point.  The early point belongs to $A(n)$ despite lying outside the disk $\B_{r+m}$, where $r = \sqrt{n/\pi}$; the late point does not belong to $A(n)$ despite lying inside the disk $\B_{r-\ell}$.}
\end {center}
\end {figure}

The next lemma gives the a priori estimate on the probability of very late points that we will use to get the proof started in Sections~\ref{sec:earlyimplieslate} and~\ref{sec:lateimpliesearly}.  It follows easily from
Lemma~6 of~\cite{LBG}; for completeness, we include a proof in appendix~\ref{sec:aprioriproof}.
	
\begin{lemma}
\label{lem:aprioriestimates}
There are absolute constants $C_0, c_0>0$ such that for all real~$\ell>0$,
\begin{equation}\label{nolate}
\PP (\Late_\ell[100\pi \ell^2]) \le C_0 e^{-c_0\ell}.
\end{equation}
%
\end{lemma}

To avoid referring to too many unimportant constants, we take $C_0$ large enough and $c_0$ small enough so that the same $C_0$ and $c_0$ work in Lemmas~\ref{lem:thickening} and~\ref{lem:aprioriestimates}.

To prove Theorem~\ref{thm:logfluctuations} we will show that for each $\gamma$ there is a constant
$a=a(\gamma)$ such that if $r$ is sufficiently large and $m= \ell = a \log r$ and $n=\pi r^2$, then
\begin{equation}\label{maingoal}
\PP(\Early_m(n) \cup \Late_\ell(n)) \leq r^{-\gamma}.
\end{equation}
The next lemma is an easy observation which shows that \eqref{maingoal} implies Theorem~\ref{thm:logfluctuations}.

\begin{lemma}
For all $\ell,m,N>0$ we have
\label{lem:earlylatealternatedefinition}
	\[ \Early_m[N]^c = \bigcap_{n \leq N} \left \{ A(n) \subset \B_{\sqrt{n/\pi}+m} \right \} \]
	\[ \Late_\ell[N]^c = \bigcap_{n \leq N} \left\{ \B_{\sqrt{n/\pi}-\ell} \subset A(n) \right\}. \]
\end{lemma}

\begin{proof}
Taking complements in the definition, we obtain
	\begin{align*}
	\Early_m[N]^c &= \bigcap_{z \in A(N)} \left\{z \not\in A(\pi(|z|-m)^2) \right\}  \\
		&= \bigcap_{z \in A(N)} \; \bigcap_{n \leq \min(N,\pi(|z|-m)^2)} \{ z \not\in A(n) \} \\
		&= \bigcap_{n \leq N} \; \bigcap_{z \,:\, |z| \geq \sqrt{n/\pi}+m} \{z \not\in A(n) \} \\
		&= \bigcap_{n \leq N} \left\{ A(n) \subset \B_{\sqrt{n/\pi}+m} \right\}.
	\end{align*}
Likewise,
	\begin{align*}
	\Late_\ell[N]^c &= \bigcap_{z \in \B_{\sqrt{N/\pi}-\ell}} \left\{z \in A(\pi(|z|+\ell)^2) \right\}  \\
		&= \bigcap_{z \in  \B_{\sqrt{N/\pi}-\ell}} \; \bigcap_{\pi (|z|+\ell)^2 \leq n \leq N} \{ z \in A(n) \} \\
		&= \bigcap_{n \leq N} \; \bigcap_{z \,:\, |z| \leq \sqrt{n/\pi}-\ell} \{z \in A(n) \} \\
		&= \bigcap_{n \leq N} \left\{ \B_{\sqrt{n/\pi}-\ell}  \subset A(n) \right\}. \qed
	\end{align*}
\renewcommand{\qedsymbol}{}
\end{proof}

\subsection{Brownian exit times}
\label{sec:exittimes}

For $x\in \R$, write $\EE_x$ for the expectation on a probability space supporting a one-dimensional standard Brownian motion $\{B(s)\}_{s \geq 0}$ started at $B(0)=x$.  For $a,b >0$, let
	\[ \tau(-a,b) = \inf \{s>0 \mid B(s)\notin [-a,b] \} \]
be the first exit time of Brownian motion from the interval $[-a,b]$.

\begin{lemma} \label{lem:exittime} Let $0 < a \le b$ and $\lambda >0$.  If $\sqrt{\lambda}(a+b) \leq 3$, then
\[
\EE_0 e^{\lambda \tau(-a,b)} \le 1 + 10\lambda ab.
\]
\end{lemma}

\begin{proof}
Let $c=(a+b)/2$, and for $x \in [-c,c]$ let
\[
f(x) = \EE_x e^{\lambda \tau(-c,c)}
\]
By translating the interval $[-a,b]$ to $[-c,c]$, it suffices to show that if $\sqrt{\lambda} c \leq 3/2$, then
	\begin{equation} \label{eq:centeredversion} f(x) \leq 1 + 10 \lambda (c+x)(c-x). \end{equation}

The function $e^{\lambda t} f(x)$ obeys the heat equation in the strip $[0,\infty) \times [-c,c]$, hence
\[
e^{\lambda t} f''(x) = - \lambda e^{\lambda t} f(x).
\]
Since $f$ is even and $f(c)=1$, we conclude that
\[
f(x) =
\frac{\cos (\sqrt{\lambda} x)}{\cos (\sqrt{\lambda} c)}.
\]

Since $\cos x$ is concave in the interval $[0,\frac32]$, we have
	\[ \cos (\sqrt{\lambda} x) \leq \cos(\sqrt{\lambda} c) + \sqrt{\lambda} (c-x) \sin (\sqrt{\lambda} c) \]
which gives
	\[ f(x) \leq 1 + \sqrt{\lambda} (c-x) \tan (\sqrt{\lambda} c). \]
Since $\tan \theta < 10 \theta$ for $\theta \in [0,\frac32]$, we conclude that for $x \in [0,c]$,
	\[ f(x) \leq 1 + 10 \lambda c (c-x) \leq 1+ 10 \lambda (c+x)(c-x) \]
which establishes \eqref{eq:centeredversion}.	
\end{proof}

The following well-known estimate appears as Proposition II.1.8 in~\cite{RY}.

\begin{lemma}\label{largedeviation}
Let $\{B(s)\}_{s \geq 0}$ be a standard Brownian motion with $B(0)=0$.  Then for any $k, s>0$
\[
\PP \left\{ \sup_{s' \in[0,s]} B(s') \ge ks \right\} \le e^{-k^2 s/2}.
\]
\end{lemma}

\subsection{The discrete harmonic function $H_\zeta(z)$} \label{sec:hzetaproperties}

Fix a lattice point $\zeta \in \Z^2$.  In this section we define the discrete harmonic function~$H_\zeta$
resembling~$F_\zeta$ in \eqref{e.fzeta}, which we will use to detect early points near $\zeta$.

For $z \in \Z^2$, let $P_n(z)$ be the probability that an $n$-step simple random walk started at the origin in $\Z^2$ ends at~$z$.  The \emph{recurrent potential kernel} for simple random walk in~$\Z^2$ is defined by
\begin{equation}
\label{eq:thepotentialkernel}
g(z) = \sum_{n=0}^\infty (P_n(0) - P_n(z)).
\end{equation}
This sum is finite for all $z \in \Z^2$, and its discrete Laplacian
	\begin{equation} \label{eq:thediscretelaplacian} \Delta g(z) = \frac{g(z+1)+g(z-1)+g(z+i)+g(z-i)}{4} - g(z) \end{equation}
vanishes for $z \neq 0$ and equals~$1$ if $z=0$.  Here we have identified $\Z^2$ with $\Z + i \Z \subset \C$.
Fukai and Uchiyama~\cite{FU} and Kozma and Schreiber~\cite{KS04} develop an asymptotic expansion for $g(z)$.  We will use only the following estimate: there are absolute constants $\lambda$ and $C_0$ (rendered explicit in~\cite{KS04}; see also~\cite[\textsection 4.4]{LL10}) such that
\begin{equation}
\label{eq:greenasymptotics}
\left| g(z) - \frac{2}{\pi} \log |z| - \lambda \right| \le C_0 |z|^{-2}.
\end{equation}

We will use this sharp asymptotic estimate of the potential kernel to show that the discrete harmonic function $H_\zeta$ defined below is close to the continuum harmonic function $F_\zeta$ (Lemma~\ref{Hsize}c), which in turn gives a discrete mean value property: the average value of $H_\zeta$ on the discrete disk $\B_r$ is very close to $H_\zeta(0)$ (Lemma~\ref{Hlevel}c).

Evidently, $g(0)=0$ and $g$ is symmetric under the dihedral
symmetries generated by the reflections across the axes
$x+iy \mapsto \pm x \pm iy$, and the diagonal, $x+iy \mapsto y+ix$.
Moreover, $g(z)>0$ for all $z\neq0$.
Exact values of $g$ near the origin can be computed using the McCrea-Whipple algorithm, described in~\cite{KS04}.  In
particular,
	\[ g(1) = 1, \qquad  g(1+i) = \frac{4}{\pi}. \]
(We will use these exact  values for convenience in specifying constants.
But what we really need is not the actual values, but the property
that $g(1+i) > g(1)$, which is not surprising since
$1+i$ is farther from the origin than $1$.)

We construct the function $H_\zeta$ for
$\zeta  = x+iy$ in the sector $S =\{\zeta \mid 0 \le y \leq x\}$.  (For $\zeta \notin S$, we choose a dihedral symmetry $\phi$ such that $\phi \zeta \in S$, and define $H_\zeta (z) := H_{\phi \zeta} (\phi z)$.)  Let
	\[ \rho = |\zeta| \]
and write
\[
	\zeta /\rho  = \alpha_1 + \alpha_2 (1+i)
\]
with $\alpha_2 = y/\rho $ and $\alpha_1 = (x-y)/\rho $.
The property we will use about the coefficients $\alpha_j$ for $j=1,2$
is that they are positive and bounded.  Indeed $0\le \alpha_j \le 1$.

Define
\[
H_\zeta(z) =  \frac{\pi}{2}[ \alpha_1g(z-\zeta  - 1) + \alpha_2g(z-\zeta  - 1-i)
- (\alpha_1+\alpha_2)g(z-\zeta ) ].
\]

Note that (using the symmetries)
\begin{align}
\label{eq:signsofHzeta}
&H_\zeta(\zeta ) = \frac{\pi}{2} [\alpha_1g(1) + \alpha_2g(1+i)] > 0 \nonumber \\
&H_\zeta(\zeta  + 1) = \frac{\pi}{2} [\alpha_2 g(1) - (\alpha_1 + \alpha_2) g(1)] < 0 \\
&H_\zeta(\zeta  + 1+i) = \frac{\pi}{2} [\alpha_1 g(1) - (\alpha_1 + \alpha_2) g(1+i)] < 0. \nonumber
\end{align}
The last inequality follows from the fact that $g(1+i)> g(1)$.

\begin {figure}[htbp]
\begin {center}
\includegraphics [width=4in]{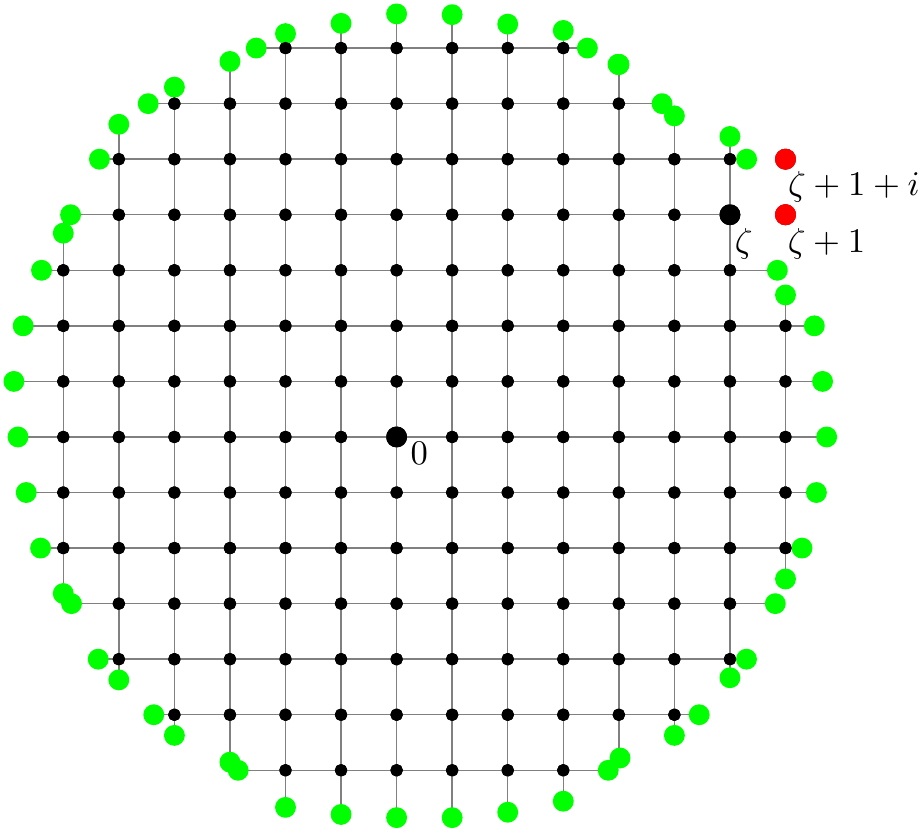}
\caption {\label{fig:hzeta} Illustration of the set $\Omega_\zeta$ for $\zeta = 6+4i$.  The function $H_\zeta$ is discrete harmonic except at the three points $\zeta$, $\zeta+1$, and $\zeta+1+i$.  Its value is positive at $\zeta$ and negative at $\zeta+1$ and $\zeta+1+i$.  For general $z \in \Grid$, we have $H_\zeta(z) = \EE H_\zeta(w)$ where~$w$ is the first of these three points hit by a Brownian motion in $\Grid$ started at~$z$.  Large green dots on the edges of $\Grid$ represent points $z \in \Grid$ for which $H_\zeta(z) = 1/2\rho$.  Together with $\zeta$, these points form the boundary of $\Omega_\zeta$.  Small black dots indicate points in $\mathbb Z^2 \cap \Omega_\zeta$.  For $z \in \Omega_\zeta$, the value $H_\zeta(z)-1/2\rho$ is proportional to the probability that a Brownian motion in $\Grid$ started at $z$ first exits $\Omega_\zeta$ at $\zeta$.}
\end {center}
\end {figure}

Extend $H_\zeta$ linearly along each edge of the grid
	\[ \Grid = \{ x+iy \mid x \in \Z \mbox{ or } y \in \Z\}. \]
(By an edge of~$\Grid$ we mean a line segment of length~$1$ whose endpoints lie in $\Z^2$.)
Now define $\Omega_\zeta $ as the connected component of the origin in the set
\[
	\{ z\in \Grid - \{\zeta \}: H_\zeta(z) > 1/2\rho \}.
\]
Recall the discrete Laplacian~$\Delta$ on~$\Z^2$, defined in \eqref{eq:thediscretelaplacian}.  We say that a function~$f$ defined on a subset $\Lambda \subset \Grid$ is \emph{grid-harmonic} if it is continuous and
	\begin{itemize}
	\item For each edge $e$ of $\Grid$, if $I \subset e \cap \Lambda$ is an interval, then $f$ is linear on $I$.
    \item For each vertex of $\mathbb Z^2$ in the interior of $\Lambda$, the sum of the slopes of the four line segments starting at $z$ (directed away from $z$) is zero.
	\end{itemize}
For example, this holds if $\Lambda = \Grid$ and $f$ is discrete harmonic on vertices and linear on edges.

\begin{lemma}\label{Hsize} ~
\begin{itemize}
\item[\em (a)] $H_\zeta$ is grid-harmonic on $\Omega_\zeta $, and $H_\zeta\ge 1/2\rho $ on $\Omega_\zeta$.
\item[\em (b)] $\zeta \in \partial \Omega_\zeta $, and for all
$z\in \partial \Omega_\zeta \backslash \{\zeta\} $ we have $H_\zeta(z) = 1/2\rho $.
\item[\em (c)] There is an absolute constant $C_1 <\infty$ such that
	\[ |H_\zeta(z) - F_\zeta(z)| \le C_1 |z-\zeta |^{-2}. \]
In particular,
\[
|H_\zeta(0) - 1/\rho | \le C_1 \rho ^{-2}.
\]
\item[\em (d)] $1 \le H_\zeta(\zeta )\le 2$.
\end{itemize}
\end{lemma}

\begin{proof}
~
\begin{itemize}
\item[(a)]  By \eqref{eq:signsofHzeta} we have $H_\zeta(\zeta + 1)< 1/2\rho$ and $H_\zeta(\zeta + 1+i)<1/2\rho$, so the points $\zeta + 1$ and $\zeta + 1+i$ are not in $\Omega_\zeta$.
Furthermore by definition $\zeta\notin \Omega_\zeta$.  Since $H_\zeta$ is
discrete harmonic at all other lattice points, its linear extension is
grid-harmonic in $\Omega_\zeta$.  Consequently, the inequality
$H_\zeta \ge 1/2\rho$ follows
from the maximum principle.

\item[(b)] This is immediate from the definition of~$\Omega_\zeta$.

\item[(c)] The approximation for $|z-\zeta| \le 10$ is trivial (choosing
$C_1$ sufficiently large).  For $|z-\zeta| > 10$, the bound \eqref{eq:greenasymptotics} implies
\[
H_\zeta(z) - \alpha_1 (\log|z-\zeta - 1| - \log |z-\zeta|)
- \alpha_2 (\log|z-\zeta - 1-i| - \log |z-\zeta|)
= O\left(\frac{1}{|z-\zeta|^2}\right).
\]
By the fundamental theorem of calculus,
\begin{align*}
\alpha_1 (\log|z-\zeta - & 1| - \log |z-\zeta|)
+ \alpha_2 (\log|z-\zeta - 1-i| - \log |z-\zeta|)  \\
&= -\alpha_1 \int_0^1 \Re \frac{1}{z-\zeta-t} dt - \alpha_2 \int_0^1 \Re \frac{1+i}{z-\zeta-t(1+i)} dt \\
&=  -\alpha_1 \Re \frac{1}{z-\zeta} - \alpha_2 \Re \frac{1+i}{z-\zeta} + O\left(\frac{1}{|z-\zeta|^2}\right) \\
&= F_\zeta(z) + O\left(\frac{1}{|z-\zeta|^2}\right).
\end{align*}

\item[(d)] Using the exact values $g(1)=1$ and $g(1+i)=4/\pi$, we have
	\[ H_\zeta(\zeta) = \frac{\pi}{2} (\alpha_1 + \frac{4}{\pi}\alpha_2) = \frac{\pi}{2}\alpha_1 + 2\alpha_2 \]
with the constraints $\alpha_i \ge 0$ and
\[
(\alpha_1 + \alpha_2)^2 + \alpha_2^2 = 1.
\]
Hence $\alpha_1 + \alpha_2 \le 1$ and
\[
 (\pi/2)\alpha_1 + 2\alpha_2 \le 2 (\alpha_1 + \alpha_2) \le 2.
\]
Thus $H_\zeta(\zeta) \le 2$.
For the lower bound, one checks that $\alpha_1$ and hence
$(\pi/2)\alpha_1 + 2\alpha_2$ is a concave function of $\alpha_2$.
At the extreme points $(\alpha_1,\alpha_2) = (1,0)$ and $(\alpha_1,\alpha_2) = (0,1/\sqrt2)$, the values of $H_\zeta(\zeta)$ are $\pi/2$ and $\sqrt2$.  Therefore,
$H_\zeta(\zeta) \ge \sqrt2 \ge 1$. $\qed$
\end{itemize}
\renewcommand{\qedsymbol}{}
\end{proof}

A corollary of Lemma \ref{Hsize} is the following.

\begin{lemma} \label{Hlevel}There is an absolute constant $C_2<\infty$ such that
\begin{itemize}
\item[\em (a)] $\B_{\rho  - C_2} \subset \Omega_\zeta  \subset B_{\rho  + C_2}$, where $B_r$ is
the disk of radius $r$ in $\R^2$.
\item[\em (b)] $1/2\rho  \le H_\zeta(z) \le 1/(\rho  - r - C_2)
\quad$ for $0\le r \le \rho  - C_2$ and $z\in \B_r$.
\item[\em (c)] For  all $0 < r \le \rho $,
\[
\left| \sum_{z\in \B_r}( H_\zeta(z)-H_\zeta(0)) \right|
+  \left| \sum_{z\in \Omega_\zeta\cap\Z^2 } (H_\zeta(z)-H_\zeta(0))\right| \le C_2 \log \rho.
\]
\end{itemize}
\end{lemma}

\begin{proof}
~
\begin{itemize}
\item[(a)] Recall that Lemma \ref{Hsize} (c) says
\[
|H_\zeta(z) - F_\zeta(z)| \le C_1 /|z-\zeta|^2.
\]
We will show that
\begin{equation}\label{outer}
F_\zeta(z) \le \frac{1}{2\rho} - \frac{C}{|z-\zeta|^2},
\quad z\in \partial B_{\rho + 2C}
\end{equation}
\begin{equation}\label{inner}
F_\zeta(z) \ge \frac{1}{2\rho} + \frac{C}{|z-\zeta|^2},
\quad z\in \partial B_{\rho - 2C}.
\end{equation}
Since the level curve $\{z:F_\zeta(z) = 1/2\rho\}$ is the circle
$\partial B_\rho\backslash\{\zeta\} $, this confirms part~(a) with $C_2=2C_1$.

To prove \eqref{outer} and \eqref{inner}, note first
that
\[
|\nabla F_\zeta (z)| = \frac{1}{|z-\zeta|^2}
\]
Indeed, by the Cauchy-Riemann equations, this is the modulus
of the (complex) derivative of $1/(\zeta-z)$.
Any curve from $\partial B_\rho$ to
$\partial B_{\rho+C}$ has length at least $C$.  In particular
following the integral curves vector field $\nabla F_\zeta$
we see that $F_\zeta$ changes by at least
\[
\frac{C}{\max_z |z-\zeta|^2}
\]
for $z$ along the curve.  These integral curves
are the circles through $\zeta$ centered on the tangent line
to $\partial B_\rho$ at $\zeta$, and $|z-\zeta|$ is
nearly constant along the portion of the integral curve
between $\partial B_\rho$ and nearby concentric circles.
Therefore the change in $F_\zeta$ is bounded as in \eqref{outer},
and similarly for \eqref{inner}.

\item[(b)] The lower bound $1/2\rho$ holds for all $z\in \Omega_\zeta$.
As we saw in (a), the level curves of $F_\zeta$ differ from
level sets of $H_\zeta$ by at most a fixed distance $2C_1$.  The
supremum of $F_\zeta$ in $B_r$ is $F_\zeta((r/\rho)\zeta) = 1/(\rho-r)$.

\item[(c)] First suppose that $r \leq \rho - 2C_1$.  By Lemma \ref{Hsize}(c),
\begin{align*}
\left| \sum_{z\in \B_r}( H_\zeta(z)-H_\zeta(0)) -
\sum_{z\in \B_r}( F_\zeta(z)-F_\zeta(0)) \right|
& \le
\sum_{z\in \B_r}
\left(\frac{C_1}{|z-\zeta|^2} + \frac{C_1}{\rho^2}\right) \\
& \le 8\pi C_1\log \rho
\end{align*}

Let $B^\Box_r$ be the union of unit squares around lattice points of $\B_r$.
Since
$|\nabla F_\zeta| \le 1/|z-\zeta|^2$, we have (writing $dz$ for area measure)
\[
\left|\int_{B^\Box_r}F_\zeta(z) dz - \sum_{z\in \B_r} F_\zeta(z) \right|
\le 4\int_{B_r}\frac{dz}{|z-\zeta|^2}  \le 8\pi \log \rho.
\]
Next,
\begin{equation}\label{annulus}
\left|\int_{B^\Box_r} F_\zeta(z)dz - \int_{B_{r}} F_\zeta(z) dz\right| \le
2\int_0^{2\pi}\left| \frac{1}{re^{i\theta} - \zeta} \right| \, r d\theta \le
8\pi \log \rho.
\end{equation}
Moreover,
\[
\left|\int_{B_{r}} F_\zeta(0) dz - \sum_{z\in \B_r}F_\zeta(0) \right|
= |\pi r^2 - \# \B_r|/\rho \le 10r/\rho \le 10.
\]
Combining all these inequalities, we have shown that
\[
\sum_{z\in \B_r}( H_\zeta(z)-H_\zeta(0))
\]
is within $C_2\log \rho$ of
\begin{equation}\label{meanval}
\int_{B_{r}}( F_\zeta(z) - F_\zeta(0))dz = 0.
\end{equation}
The last equation \eqref{meanval} holds because $F_\zeta$ is harmonic in $B_{r}$.

We now turn to the cases of $\Omega_\zeta$ and  $\B_r$ with $\rho - 2C_1 < r \leq \rho$.  This follows from the case $\B_{\rho -2C_1}$, which we have already proved. In fact,
these other sums differ only on lattice sites in $\rho - 2C_1 \le |z| \le \rho + 2C_1$, and the sum over these sites is majorized the same way as in \eqref{annulus},
	\[ \sum_{z \in \B_{\rho+2C_1} \backslash \B_{\rho-2C_1}} |H_\zeta(z)| \leq
	\sum_{z \in \B_{\rho+2C_1} \backslash \B_{\rho-2C_1}} \min(\frac{2}{|z-\zeta|}, 4) \leq 32\pi C_1 \log \rho. \qed \]
\end{itemize}
\renewcommand{\qedsymbol}{}
\end{proof}

\begin{remark}
One can replace
$H_\zeta$ with harmonic functions
for which Lemma \ref{Hlevel}(a) and~(b) hold with~$B_r$ replaced by shapes other than disks.  However,  the approximate
mean value property Lemma \ref{Hlevel}(c) relies on~$B_r$ being a disk.  The main
ingredient is the continuum mean value property \eqref{meanval},
which can be rephrased as saying that harmonic measure on
$\partial B_r$ from the origin is the same as from a uniform
point in $B_r$.  Accordingly, the place where the approximate mean value property
Lemma~\ref{Hlevel}(c) is used marks a crucial step in the proof: The
bound for $\B_r$ is used in the proof of
Lemma~\ref{lem:earlyimplieslate}, and the bound for $\Omega_\zeta$ in
the proof of Lemma~\ref{lem:lateimpliesearly}.
\end{remark}

\subsection{The martingale $M_\zeta(t)$}
\label{sec:martingale}

For each integer $n$, recall that the set $A(n)$ consists of
$n$ lattice points with $A(n)\subset A(n+1)$ growing according to
the IDLA scheme.
For $\zeta \in \Z^2$, we will define a modified cluster $A_\zeta(n)$ in which particles are stopped on exiting $\Omega_\zeta$; moreover, we will define $A_\zeta(t)$ for all $t\in \R$, not just for integer values.

To define the modified cluster, recall first the notion of Brownian motion $\beta(s)$ on the grid
	\[ \Grid = \{x+iy \in \C \mid x\in \Z \mbox{ or } y \in \Z \}. \]
To construct $\beta(s)$,
take a standard Brownian motion $B(s)$ starting at $0$ at time $s_0=0$
until the first time $s=s_1$ that $B(s)= \pm 1$. For each open interval of the set $S_1 = \{s: 0 < s < s_1,
\  B(s) \neq 0\}$ choose a direction (up, down, left, right)
with equal probability $1/4$, independently.  Define a corresponding
motion of a particle $\beta(s) = B(s)$ in an interval in which right was chosen
(and $\beta(s)= B(s)i$ in an interval in which up was chosen, and
similarly for left and down).  On each interval except the final one, the particle returns to the origin.
On the final interval it reaches a lattice point adjacent to the origin for the first time.  We continue the process (centered at this new lattice point) until it hits a point adjacent to that one, etc.

Let $\widetilde{\beta}_0(s), \widetilde{\beta}_1(s), \ldots$ be independent Brownian motions on the grid $\Grid$, and for each integer $n\geq 0$, let
	\[ s^*_n = \inf \left\{ s \geq 0 \mid \widetilde{\beta}_n(s) \in (\Z^2 - A_\zeta(n)) \cup (\Grid - \Omega_\zeta) \right\} \]
be the first time that $\widetilde{\beta}_n$ visits either a lattice point outside the cluster $A_\zeta(n)$, or any point outside the set $\Omega_\zeta$.
Then for $0 \leq s \leq 1$, define
	\[ A_\zeta(n+s) = A_\zeta(n) \cup \{ \beta_{n+1}(s) \} \]
where
	\[ \beta_n(s) = \widetilde{\beta}_n \left( \min \left( \frac{s}{1-s},\, s^*_n \right) \right), \qquad 0 \leq s < 1 \]
(and by definition, $\beta_n(1) = \widetilde{\beta}_n(s^*_n)$).  The choice of $\frac{s}{1-s}$ is not important here; for our purposes, any increasing function $f(s)$ satisfying $f(0)=0$ and $f(s) \uparrow \infty$ as $s \uparrow 1$ would suffice in its place.  Effectively, the particle moves at an ever-increasing rate so that it \emph{a.s.} visits the set $(\Z^2 - A_\zeta(n)) \cup (\Grid - \Omega_\zeta)$ in less than unit time.

The points of $\partial \Omega_\zeta $ where (eventually) many particles accumulate include~$\zeta $, but are typically non-lattice points on the grid~$\Grid$.
We regard $A_\zeta(t)$ as a multiset: it may have points of multiplicity greater than one on the boundary $\partial \Omega_\zeta$, or at interior vertices at times when the moving particle is on top of an occupied site.
When summing a function over $A_\zeta(t)$, we take these multiplicities into account; thus for a function $f$ defined on $\Omega_\zeta \cup \partial \Omega_\zeta$ and $t \in [n,n+1)$, the notation $\sum_{z \in A_\zeta(t)} f(z)$ means,
	\[ \sum_{z \in A_\zeta(t)} f(z) := \sum_{j=0}^{n-1} f(\beta_j(1)) + f(\beta_{n}(t-n)). \]

Recall the function $H_\zeta$ defined in Section~\ref{sec:exittimes}.  Because $H_\zeta$ is grid-harmonic in $\Omega_\zeta $ (that is, discrete harmonic at lattice points, and linear on edges) , the process
\[
M_\zeta(t) := \sum_{z\in A_\zeta(t)} (H_\zeta(z)-H_\zeta(0))
\]
(summing with multiplicity) is a continuous-time martingale adapted to the filtration
	\[ \mathcal{F}_t = \sigma \{ A_\zeta(s) \mid 0 \leq s \leq t \}. \]
We will refer to $M_\zeta$  as the martingale with pole at $\zeta$.

By the martingale representation theorem of Dambis and Dubbins-Schwarz (see, e.g.~\cite[Theorem~V.1.6]{RY}),
\[
M_\zeta(t) = B_\zeta(S_\zeta(t))
\]
where $B_\zeta(s)$ is a standard Brownian motion and $S_\zeta(t) = \langle M_\zeta, M_\zeta \rangle_t$ is the quadratic variation of~$M_\zeta$.

Note that $S_\zeta(n)$ is a stopping time with respect to the filtration $\{\mathcal{F}_{T_\zeta(s)} \}_{s \geq 0}$, where $T_\zeta(s) = \inf \{t \mid S_\zeta(t) > s \}$; and $B_\zeta(s)$ is adapted to this filtration.
By the strong Markov property (\cite[Theorem~2.16]{MP}), for any $n$, the process
	\[ \widetilde{B}^n_\zeta(u) := B_\zeta(S_\zeta(n)+u) - B_\zeta(S_\zeta(n)) \]
is a standard Brownian motion that is independent of $\mathcal{F}_n$.
It follows that we can write
	\[ B_\zeta(s) = B_\zeta(S_\zeta(n)) + \widetilde{B}^{n}_\zeta (s - S_\zeta(n)), \qquad  s \in [S_\zeta(n), S_\zeta(n+1)]. \]
where $\widetilde{B}^0_\zeta(\cdot), \widetilde{B}^1_\zeta(\cdot), \ldots$ are independent standard Brownian motions started at~$0$.

For $-a < 0 < b$, let
	\[ \tau_n(-a,b) = \inf \{u > 0 \mid \widetilde{B}^n_\zeta(u) \notin [-a,b] \} \]
be the first exit time of $\widetilde{B}^n_\zeta$ from the interval $[-a,b]$.

\begin{lemma}
\label{lem:boundedbyexittime}
Fix $\zeta \in \Z^2$.  Let
	\[ -a_n = \min_{z \in \partial A_\zeta(n)} (H_\zeta(z)-H_\zeta(0)) \]
	\[ b_n = \max_{z \in \partial A_\zeta(n)} (H_\zeta(z)-H_\zeta(0)). \]
Then
	\[ S_\zeta(n+1) - S_\zeta(n) \leq \tau_{n}(-a_n,b_n). \]
\end{lemma}

\begin{proof}
By the maximum principle, since $H_\zeta$ is harmonic on $A_\zeta(n) \subset \Omega_\zeta$, we have
	 \[ -a_n \leq H_\zeta(z) - H_\zeta(0) \leq b_n \]
for all $z \in A_\zeta(n)$.  Hence for all $t \in [n,n+1]$, we have
	\[ -a_n \leq M_\zeta(t) - M_\zeta(n) \leq b_n. \]
That is, for all $s \in [S_\zeta(n), S_\zeta(n+1)]$, we have
	\[ -a_n \leq B_\zeta(s) - B_\zeta(S_\zeta(n)) \leq b_n. \]
That is, for $u \in [0, S_\zeta(n+1)-S_\zeta(n)]$, we have
	\[ -a_n \leq \widetilde{B}^n_\zeta(u)  \leq b_n. \]
Hence the first exit of $\widetilde{B}^n$ from the interval $[-a_n,b_n]$ occurs after time $S_\zeta(n+1) - S_\zeta(n)$, which gives $S_\zeta(n+1) - S_\zeta(n) \leq \tau_n(-a_n,b_n)$.
\end{proof}

Fix $m\geq 1$.  To detect the presence of $m$-early points in the aggregate $A(t)$, we will use the martingales $M_\zeta$ for $\zeta$ lying just outside the aggregate.  The next lemma helps control the elapsed time $S_\zeta(t)$ on the event that there are no $(m+1)$-early points.  Define
\[
r_0(t) = \sqrt{t/\pi} + 4m + 2C_2.
\]

\begin{lemma}
\label{lem:earlyquadraticvariation}
For all $t>0$ and all $\zeta \in \Z^2$ such that $|\zeta| \geq r_0(t)$,
\begin{equation}\label{Sbound}
\EE \left[ e^{S_\zeta(t)} \one_{\Early_{m+1}[t]^c} \right] \leq t^{40}.
\end{equation}
\end{lemma}

\begin{proof}
Fix $t$, and abbreviate $r_0 = r_0(t)$.
On the event $\Early_{m+1}[t]^c$, we have for all $n \leq t$
	\[ A(n) \subset \B_{\sqrt{n/\pi}+m+1}. \]
Hence by Lemma \ref{Hlevel}(b), on the event $\Early_{m+1}[t]^c$ we have
	\[ -\frac{1}{2r_0} \leq H_\zeta(z) - H_\zeta(0) \leq \frac{1}{r_0 - \sqrt{n/\pi} - m -2- C_2} \]
for all positive integers $n \leq t$ and all $z \in A_\zeta(n)$.

We obtain from Lemma~\ref{lem:boundedbyexittime},
\[
(S_\zeta(n) - S_\zeta(n-1))\one_{\Early_{m+1}[t]^c}
\le \tau_n \left(-a,b_n \right).
\]
where \[ a = \frac{1}{2r_0}, \qquad b_n=\frac{1}{r_0 - \sqrt{n/\pi} - m -2- C_2}. \]
Summing over $n=1,\ldots,t$ yields
\[
e^{S_\zeta(t)} \one_{\Early_{m +1}[t]^c} \leq e^{S_\zeta(t) \one_{\Early_{m+1}[t]^c}} \leq e^{\sum_{n=1}^{t} \tau_n(-a,b_n)}.
\]
The Brownian exit times $\tau_n(-a,b_n)$ for $n=1,\dots, t$ are independent.  Using Lemma~\ref{lem:exittime} with $\lambda=1$, we obtain
	\[ \log \EE e^{\sum_{n=1}^{t} \tau_n(-a,b_n)} =  \sum_{n=1}^{t}  \log \EE e^{\tau_n(-a,b_n)}  \le \sum_{n=1}^{t}  \log \left( 1 + 10ab_n \right)  \le
\sum_{n=1}^{t} 10ab_n. \]
Setting $r_1 = r_0 - m - 3 - C_2$ we have
	\begin{align*} \sum_{n=1}^t b_n &\leq \int_{0}^t \frac{dn}{r_1 - \sqrt{n/\pi}} \\
							  &= \int_{0}^{\sqrt{t/\pi}} \frac{2\pi x \,dx}{r_1 -x} \\
							  &= 2\pi \int_{r_1 - \sqrt{t/\pi}}^{r_1} \frac{r_1-y}{y} dy \\
					  &\leq 2\pi r_1 \log \frac{r_1}{r_1 - \sqrt{t/\pi}} \leq 2\pi r_0 \log \frac{r_0}{C_2}.
	\end{align*}
Hence $\sum_{n=1}^t 10ab_n \leq 10 \pi \log(r_0/C_2)$.  We conclude that
	\[ \EE \left[ e^{S_\zeta(t)} \one_{\Early_{m+1}[t]^c} \right] \leq (r_0/C_2)^{10 \pi} \leq t^{40}. \qed \]
\renewcommand{\qedsymbol}{}
\end{proof}

In order to study late points in the cluster $A(t)$, we will need a variant of Lemma~\ref{lem:earlyquadraticvariation} which helps control the elapsed time $S_\zeta(t)$ when $|\zeta|$ is slightly smaller than $r_0(t)$.  In this case the quadratic variation of $M_\zeta$ is larger due to particles that accumulate on the boundary $\partial \Omega_\zeta$.

\begin{lemma}
\label{lem:latequadraticvariation}
Fix $m \geq 2C_2 + 2$ and $\ell \leq m$.   Fix $\zeta \in \Z^2$, and let $t = \pi(|\zeta| + \ell)^2$.  If $|\zeta | \geq \ell$, then
\begin{equation}
\EE \left[ e^{S_\zeta(t)} \one_{\Early_m[t]^c} \right] \leq t^{40} e^{700 m}.
\end{equation}
\end{lemma}

\begin{proof}
Let $\rho = |\zeta |$, and suppose first that $\rho \geq 5m$.
Let $t_0 = \floor{\pi(\rho - 4m - 2C_2 - 1)^2}$.  Then $r_0(t_0) \leq \rho$, so by Lemma~\ref{lem:earlyquadraticvariation},
	\[ \EE \left[ e^{S_\zeta(t_0)} \one_{\Early_m[t_0]^c} \right] \leq \EE \left[ e^{S_\zeta(t_0)} \one_{\Early_{m+1}[t_0]^c} \right] \leq t_0^{40} \leq t^{40}. \]
(Here we have used the trivial inclusion $ \displaystyle \Early_m[t_0]^c \subset \Early_{m+1}[t_0]^c.$)

Furthermore, because $H_\zeta\le 2$ on all of $\partial \Omega_\zeta$, we have by Lemma~\ref{lem:boundedbyexittime} for $n=t_0 +1,\ldots,t$
\[
S_\zeta(n) - S_\zeta(n-1) \leq \tau_n(-\frac{1}{2\rho},2).
\]
Hence, using the inclusion $\Early_m[t]^c \subset \Early_m[t_0]^c$,
\[
e^{S_\zeta(t)}\one_{\Early_{m}[t]^c} \leq
e^{S_\zeta(t)\one_{\Early_{m}[t_0]^c}} \leq
e^{S_\zeta(t_0)\one_{\Early_{m}[t_0]^c}} \prod_{n=t_0+1}^{t} e^{\tau_n(-\frac{1}{2\rho}, 2)}.
\]
By Lemma~\ref{lem:exittime}, we have $\log \EE e^{\tau_n(-\frac{1}{2\rho},2)} \leq \frac{10}{\rho}$ for $t_0 < n \leq t$, hence
\[
\sum_{t=t_0+1}^{t} \log \EE e^{\tau_n} \leq \frac{10}{\rho}(t  - t_0)
\le 400m.
\]
In the last inequality we have used that $t \leq \pi (\rho+ m)^2$ and $m \geq 2C_2+2$.

Since the exit times $\tau_n$ for $n=t_0+1,\dots, t$ are independent of one another and of $\mathcal{F}_{t_0}$, we obtain
\begin{align*}
\log \EE \left[ e^{S_\zeta(t)}\one_{\Early_{m}[t]^c} \right] &\leq
\log \EE \left[ e^{S_\zeta(t_0)\one_{\Early_{m}[t_0]^c}} \right] + \sum_{n=t_0+1}^{t} \log \EE e^{\tau_n} \\
	&\leq 40 \log t + 400m.
\end{align*}

It remains to consider the case $\ell \leq \rho \leq 5m$.  In this case we take $t_0 = 0$, obtaining
	\[ \log \EE e^{S_\zeta(t)} \leq \frac{10}{\rho} t = \frac{10 \pi (\rho+\ell)^2}{\rho} \leq  40 \pi \rho \leq 700m. \qed \]
\renewcommand{\qedsymbol}{}
\end{proof}

\subsection{Early points imply late points}
\label{sec:earlyimplieslate}

Now we begin the main part of the proof of Theorem~\ref{thm:logfluctuations}.
Fix $\gamma \geq 1$.
With constants $b$, $C_0$ and $c_0$ given by Lemma~\ref{lem:thickening}, and $C_2$ given by Lemma~\ref{Hsize}, we set
\[
C_3 =  \max \left\{ (\gamma + 5 + \log C_0)/c_0, \, (1000/b) C_2 \right\}.
\]
(In particular, if $T \geq 3$ and $m\geq C_3\log T$, then $C_0 e^{-c_0m} < T^{-(\gamma+5)}$.)

Recall the events
	\[ \Early_m[T]  = \Big \{ \mbox{some $z \in A(T)$ is $m$-early} \Big\} \]
and
	\[\Late_\ell[T] = \Big \{ \mbox{some $z \in \B_{\sqrt{T/\pi}-\ell}$ is $\ell$-late} \Big\} \]
defined in Section~\ref{sec:basecase}.

\begin{lemma} (Early points imply late points)
\label{lem:earlyimplieslate}
Fix $\gamma \geq 1$, $T \geq 20$ and $m \geq C_3 \log T$.  If $\ell \leq (b/1000)m$, then
	\[ \PP( \Early_m[T] \cap \Late_{\ell}[T]^c ) \leq T^{-(\gamma+1)}. \]
\end{lemma}

In other words, with high probability either there is no $m$-early point by time~$T$, or there is a $(b/1000)m$-late point by time~$T$.

\begin{proof}
The proof consists of two main estimates.  First, if $z \in \Z^2$ is the first $m$-early point and joins the cluster at time $t$, then for a suitably chosen point $\zeta$ near $z$, the quadratic variation $S_\zeta(t)$ is small: equation \eqref{Sbound2.1} below.  Second, if in addition there are no $\ell$-late points by time $t$, then the martingale $M_\zeta(t)$ is large: equation \eqref{farfrommean} below.  To finish the proof we use the large deviation bound for Brownian exit times to show that these two events (small $S_\zeta(t)$ and large $M_\zeta(t)$) are unlikely both to occur.

Without loss of generality we assume that $\ell = (b/1000) m$.
Note that $\Early_m[T]$ is empty if $m> T^{1/2}$, so we may also
assume $m\le T^{1/2}$.
For each integer $t=1,2,\dots, T$ and each lattice point $z$,
let
	\[ Q_{z,t} = \left\{z \in A(t) - A(t-1)\} \cap \{z \in A(\pi(|z|-m)^2) \right\} \cap \Early_{m}[t-1]^c  \]
be the event that $z$ first joins the cluster at time~$t$ and $z$ is $m$-early, but no previous point is $m$-early.  Since $A(T) \subset \B_T$,
	\begin{equation} \label{eq:earlydisjointunion} \bigcup_{t \leq T} \bigcup_{z \in \B_T} Q_{z,t} =
 \Early_m[T]. \end{equation}


As in Lemma \ref{lem:earlyquadraticvariation}, we write
\[
r_0 = r_0(t) = \sqrt{t/\pi} + 4m + 2C_2.
\]
Fix a lattice point $z \in \B_t$, and consider the unit square with lattice point corners
containing $r_0 z/|z|$.  Define $\zeta := \zeta_0(z,t)$ to be the
corner of that square farthest from the origin.  Then
\[
 r_0\le |\zeta_0(z,t)| < r_0 + 2.
\]

By Lemma~\ref{lem:earlyquadraticvariation}, we have for $\zeta = \zeta_0(z,t)$
	\[ \EE e^{S_\zeta(t)} \one_{\Early_{m+1}[t]^c} \leq T^{40}. \]
Hence by Markov's inequality, letting $s = (2\gamma+100) \log T$, we have
	\[ \PP( \Early_{m+1}[t]^c \cap \{ S_{\zeta}(t) > s \} ) \leq T^{-(2\gamma+100)} \EE e^{S_\zeta(t)} \one_{\Early_{m+1}[t]^c} \leq T^{-(2\gamma+60)}. \]

On the event $Q_{z,t}$, there are no $m$-early points before time $t$; moreover, $z$ is at unit distance
from some point of $A(t-1)$, so $z$ is not $(m+1)$-early.  Therefore,
\[
Q_{z,t}  \subset \Early_{m+1}[t]^c.
\]
and hence
\begin{equation}\label{Sbound2.1}
\PP(Q_{z ,t } \cap \{ S_\zeta(t ) > s \}) \leq T^{-(2\gamma+60)}.
\end{equation}

Next, we show that for $\zeta = \zeta_0(z ,t )$,
\begin{equation}\label{farfrommean}
\PP(\{M_\zeta(t ) < (b/10)m\} \cap \Late_\ell[t ]^c \cap Q_{z ,t }) < T^{-(\gamma+5)}.
\end{equation}
In other words, on the event $E := \Late_\ell[t ]^c \cap Q_{z ,t }$ it
is likely that $M_\zeta(t ) \ge (b/10) m$.

On the event $Q_{z ,t }$,
no $z'\in A(t )$ is $(m+1)$-early, so we have
$|z'| \le \sqrt{t/\pi} + m + 1 < r_0 - C_2$ for all $z' \in A(t)$.  In other words,
\[
A(t ) \subset \B_{r_0 - C_2}.
\]
Hence Lemma \ref{Hlevel}(a) implies that $A_\zeta(t )$ does
not meet $\partial \Omega_\zeta$.  In other words, on the event $Q_{z,t}$ we have
	\[ A_\zeta(t) = A(t). \]
Partition $A_\zeta(t)$ into three disjoint sets as follows.
\[
A_1 = A_\zeta(t) \cap \B_{\sqrt{t/\pi} - \ell} ; \quad
A_2 = A_\zeta(t) \cap \B (z,m); \quad
A_3 = A_\zeta(t) \setminus (A_1 \cup A_2).
\]
On the event $E$, the ball of radius
$\sqrt{t/\pi} -\ell $ is filled by $A_\zeta(t)$, so
\[
A_1 = \B_{\sqrt{t/\pi} - \ell};
\]
and since $A_\zeta(t)$ has $t$ points, $A_2 \cup A_3$ has at
most $10 \sqrt {t} \ell $ points.   In particular,
$(\# A_3)/r_0 \le 10\sqrt{\pi} \ell$ and
\[
\sum_{z'\in A_3} (H_\zeta(z') - H_\zeta(0))
\ge - (\# A_3) H_\zeta(0) \ge - 20 \ell = -(b/50) m,
\]
since $H$ is non-negative on $\Omega_\zeta$ and, by Lemma~\ref{Hsize}(c),
$H_\zeta(0)$ is close to $1/r_0$.
Furthermore, Lemma \ref{Hlevel}(c) implies
\[
\sum_{z'\in A_1} (H_\zeta(z') - H_\zeta(0))
\ge -  C_2 \log r_0 \ge -C_2 \log T > -(b/1000) m,
\]
where in the last step we have used that $C_2 < bC_3/1000$ and $C_3 \log T \leq m$.
On the other hand, since $|z|\ge m$, it follows from
Lemma~\ref{lem:thickening} that
\[
\PP(\{\#A_2 \le  bm^2\} \cap E) <  C_0 e^{-c_0 m} \le T^{-(\gamma+5)}.
\]
Next, recall that on the event $Q_{z,t}$, the point $z$ is $m$-early but not $(m+1)$-early, so that
$\sqrt{t /\pi} + m \le |z | \le \sqrt{t /\pi} + m + 1$.  Moreover,
$\zeta$ is at distance roughly $3m$ from $z $ along the ray
$0$ to $z $.  $F_\zeta(z ) \approx 1/3m$, and using Lemma \ref{Hsize}(c) we find that $H_\zeta(z') \ge 1/4m + O(1/m^2)$ for all $z'\in \B (z,m)$.
Furthermore, $H_\zeta(0) \approx 1/r_0$ and $r_0 \ge \sqrt{t /\pi} \ge 10m$,
so $H_\zeta(z') - H_\zeta(0) \ge 1/8m$.  Therefore, on the event $\{\#A_2 >  bm^2\} \cap E$, we have
\[
\sum_{z'\in A_2} (H_\zeta(z') - H_\zeta(0))
\ge (\# A_2) \frac{1}{8m} >  b m^2\frac{1}{8m} = (b/8)m.
\]
Summing the contributions to $M_\zeta(t )$ from $A_1$, $A_2$ and $A_3$, we find that on the event $\{\#A_2 >  bm^2\} \cap E$ we have $M_\zeta(t) > (\frac{b}{8} - \frac{b}{50} - \frac{b}{1000})m > (b/10)m$.  Hence
	\[ \{ M_\zeta(t) \leq (b/10)m \} \subset \{\# A_2 \leq bm^2 \} \cup E^c. \]
Intersecting with $E$, we obtain
\[
\PP (\{M_\zeta(t ) \leq (b/10)m \} \cap E) \leq \PP(  \{\# A_2 \leq bm^2 \} \cap E)  < T^{-(\gamma+5)}
\]
which proves \eqref{farfrommean}.

\old{
From \eqref{eq:earlydisjointunion} we have $\bigcup_{t=1}^{100\pi m^2} \bigcup_z Q_{z,t} = \Early_m [100 \pi m^2]$.
Since the events $Q_{z,t}$ are disjoint for $(z,t) \neq (z',t')$, we obtain from \eqref{eq:earlydisjointunion} and Lemma~\ref{lem:aprioriestimates}
\begin{equation}\label{earlytimes}
\sum_{t=1}^{100\pi m^2} \sum_{z} \PP(Q_{z,t})
= \PP(\Early_{m}[100 \pi m^2])
\le C_0 e^{-c_0m} < T^{-n-2}.
\end{equation}
}

Recall that $s = (2\gamma+100)\log T$.   Since $m \geq C_3 \log T$ we have $(b/10)m \geq s$.  Using the large deviation bound for the maximum of Brownian motion, Lemma \ref{largedeviation}, we obtain
	\begin{align} \PP \big( \left\{ S_\zeta(t) \leq s \right\} \cap \left\{ M_\zeta(t) \geq (b/10)m \right\} \big)
	&\leq \PP \big( \left\{ S_\zeta(t) \leq s \right\} \cap \left\{ B_\zeta(S_\zeta(t)) \geq s \right\} \big) \nonumber \\
	&\leq \PP \left\{ \sup_{s' \in [0,s]} B_\zeta(s') \geq s \right\} \nonumber \\
	&\leq e^{-s/2} = T^{-(\gamma+50)}. \label{eq:nolargemean}
	\end{align}

Finally, we assemble these ingredients to bound
\begin{align*}
\PP(\Early_m[T] \cap \Late_\ell[T]^c) =
 \sum_{t  = 1}^T \sum_{z \in \B_T} \PP(Q_{z ,t } \cap \Late_\ell[T]^c).
\end{align*}
Split each term in the sum on the right into three pieces, with $\zeta = \zeta_0(z ,t )$,
as indicated below and majorize using \eqref{Sbound2.1}, \eqref{farfrommean},
and  \eqref{eq:nolargemean}, respectively.  We obtain with $s=(2\gamma+100) \log T$,
\begin{align*}
\PP(Q_{z ,t } \cap \Late_\ell[T]^c)
&\le \PP (Q_{z ,t }\cap \{S_\zeta(t ) > s \}) \\
& \qquad+ \PP (Q_{z ,t } \cap \{M_\zeta(t ) < (b/10)m \} \cap \Late_\ell[t ]^c) \\
& \qquad \qquad + \PP ( \{S_\zeta(t ) \le s \} \cap \{ M_\zeta(t ) \ge (b/10) m \}) <3 T^{-(\gamma+5)}
\end{align*}
Summing over $z\in \B_T$ and $t=1,\ldots,T$, since $T \geq 20$ we conclude that
	\[ \PP(\Early_m[T] \cap \Late_\ell[T]^c) \leq  (2\pi T^2) \cdot T \cdot 3T^{-(\gamma+5)}
	 < T^{-(\gamma+1)}. \]
This completes the proof of Lemma~\ref{lem:earlyimplieslate}.
\end{proof}

\subsection{Late points imply early points}
\label{sec:lateimpliesearly}

Fix $\gamma \geq 1$, and let
\[
C_4 = (\gamma+3) (1500 + \frac{1}{c_0}) +  \frac{2\log C_0}{c_0} +  2C_2 + 2.
\]
where $C_2$ is the constant appearing in Lemma \ref{Hlevel}, and~$C_0$ and~$c_0$ appear in Lemma~\ref{lem:aprioriestimates}.
\begin{lemma} (Late points imply early points)
\label{lem:lateimpliesearly}
If $\ell \geq C_4\log T$ and $m \leq \ell^2/ C_4\log T$, then
	\[ \PP(\Late_\ell[T] \cap \Early_m[T]^c) \leq T^{-(\gamma+1)}. \]
\end{lemma}
In other words, with high probability, either there is
no $\ell$-late point by time~$T$, or there is an
$\ell^2 /C_4 \log T$-early point by time~$T$.

\begin{proof}

Without loss of generality we take $m=  \ell^2/ C_4\log T$.  Note that since $\ell \geq C_4 \log T$ this ensures $m \geq \ell$.
Let $\zeta\in \Z^2$ be such that $|\zeta| \leq \sqrt{T/\pi} - \ell$, and set
\[
\rho  = |\zeta |, \quad   T_1  = \pi(\rho + \ell)^2.
\]
Note that $T_1 \leq T$.  Denote by $L[\zeta] = \{\zeta \not\in A(T_1) \} $ the event that $\zeta $ is $\ell$-late.  Then
\[
\Late_\ell[T] = \bigcup_{|\zeta| \leq \sqrt{T/\pi}-\ell} L[\zeta].
\]

We will use the martingale $M_\zeta (t)$ with pole at $\zeta $.
We first show that on the event~$L[\zeta ]$,
\begin{equation}\label{M0bound}
M_\zeta (T_1) \leq -\ell.
\end{equation}
To prove this, note first that by the maximum principle, the smallest value
of~$H_\zeta$ on~$\Omega_\zeta$ is $1/2\rho $ attained on the boundary $\partial \Omega_\zeta $.  This
boundary value is achieved everywhere on $\partial \Omega_\zeta$ except at the point $\zeta $, at which $H_\zeta$ has
a much larger value.  On the event $L[\zeta]$, no particle reaches $\zeta $ by time $T_1 $, so
the values of $H_\zeta$ at all boundary points reached at time $T_1$
are $1/2\rho $, which is smaller than the interior values.  Therefore,
$M_\zeta (T_1)$ is maximized if all the sites of $\Omega_\zeta $ are occupied, i.e.,
$\Omega_\zeta \subset A_\zeta(T_1)$.  By Lemma~\ref{Hlevel}(a), we have $\Omega_\zeta \subset \B_{\rho+C_2}$, so
	\[ T_1 - \# \Omega_\zeta \geq \pi (\rho+\ell)^2 - \pi(\rho+C_2+1)^2. \]
Since $\ell \geq C_4 \log T$, the right side is at least $\pi \rho \ell$.  Hence on the event~$L[\zeta]$, we have
	\begin{equation}
	\label{eq:boundaryplusinterior}
	M_\zeta(T_1) \leq
	(\pi \rho \ell )\left(\frac{1}{2\rho } - H_\zeta(0)\right)
		+ \sum_{z\in \Omega_\zeta } (H_\zeta(z) - H_\zeta(0)).
	\end{equation}
By Lemma~\ref{Hsize}(c), we have $\frac{1}{2\rho} - H_\zeta(0) = -\frac{1}{2\rho}  + O(\frac{1}{\rho^2})$, which shows that the first term on the right side of \eqref{eq:boundaryplusinterior} is less than $-3\ell/2$.
By Lemma \ref{Hlevel}(c), the second term is at most $C_2 \log \rho \leq C_2 \log T \leq \ell/2$.
This concludes the proof of \eqref{M0bound}.

Now suppose that $|\zeta| \geq \ell$.  Recall that $m \geq \ell \geq C_4 \log T$ and $C_4 \geq 1$, so $e^m \geq T$.  By Lemma~\ref{lem:latequadraticvariation},
\begin{equation}\label{Sbound3}
\EE e^{S_\zeta(T_1)} \one_{\Early_m[T_1 ]^c} \leq T^{40} e^{700 m} \leq e^{740m}.
\end{equation}
Letting $s=750m$, we have by Markov's inequality
\[
\PP(\{ S_\zeta(T_1) > s \} \cap \Early_m[T_1]^c )
\leq e^{-s} \EE e^{S_\zeta(T_1)} \one_{\Early_m[T_1]^c}
\leq e^{-10m}.
\]
Now by \eqref{M0bound},
	\begin{align*} \PP ( L[\zeta] \cap \Early_m[T]^c )
	&\leq \PP( \Early_m[T]^c \cap \{S_\zeta(T_1)>s \} )
		+ \PP ( \{S_\zeta(T_1) \leq s \} \cap L[\zeta] ) \\
	&\leq e^{-10m} + \PP \big\{S_\zeta(T_1) \leq s, \, M_\zeta(T_1) \leq -\ell \big\}.
	\end{align*}
We use Lemma~\ref{largedeviation} to bound the second term.  Since $M_\zeta(t) = B_\zeta(S_\zeta(t))$, where $B_\zeta$ is a standard Brownian motion with $B_\zeta(0)=0$, we have
	\begin{align*}
	\PP \left\{S_\zeta(T_1) \leq s, \, M_\zeta(T_1) \leq -\ell \right\}
	\leq \PP \left\{ \inf_{s' \in [0,s]} B_\zeta(s') \leq -\ell \right\}
	\leq e^{- \ell^2 /2s }.
	\end{align*}
We conclude that for all $\zeta \in \Z^2$ such that $ |\zeta | \leq \sqrt{T/\pi}-\ell$,
	\begin{equation*} \PP ( \Early_m[T]^c \cap L[\zeta] ) \leq e^{-10m} + e^{-\ell^2 / 1500m}. \end{equation*}
Since $m \geq \ell \geq (\gamma+3)\log T$ and $\ell^2 / 1500m \geq (\gamma+3) \log T$ by hypothesis, each term on the right side is at most $T^{-(\gamma+3)}$.  Hence for $\ell \leq |\zeta| \leq \sqrt{T/\pi}-\ell$, we have
	\begin{equation} \label{eq:sum.me} \PP ( \Early_m[T]^c \cap L[\zeta] ) \leq 2T^{-(\gamma+3)}. \end{equation}

To take care of the case $|\zeta| < \ell$, we use Lemma~\ref{lem:aprioriestimates}, which gives
\begin{equation*}\label{smallzeta}
\PP\left(\bigcup_{|\zeta| < \ell} L[\zeta]\right)
=  \PP(\Late_\ell [4\pi \ell^2])
\le C_0 e^{-c_0\ell} < T^{-(\gamma+3)}.
\end{equation*}

Summing \eqref{eq:sum.me} over $\zeta \in \B_{\sqrt{T/\pi}-\ell} - \B_\ell$ yields
	\begin{align*}
	\PP(\Early_m[T]^c \cap \Late_\ell[T])
	&\leq \PP \left(\bigcup_{|\zeta| < \ell} L[\zeta]\right) + \sum_{\ell \leq |\zeta | \leq \sqrt{T/\pi}-\ell} \PP(\Early_m[T]^c \cap L[\zeta]) \\
	&< T^{-(\gamma+3)} + T^{-(\gamma+2)}.
	\end{align*}
This completes the proof of Lemma~\ref{lem:lateimpliesearly}.
\end{proof}

The proof of the main result from
Lemmas~\ref{lem:earlyimplieslate} and~\ref{lem:lateimpliesearly}
is a routine iteration.  We use the lemmas in their contrapositive direction
(no early point implies no late point and no late point implies no early point)
for progressively smaller values of $m$ and $\ell$, starting from \eqref{nolate}.

Fix $\gamma \geq 1$ and $T\geq 20$. Let $\ell_0 = \sqrt{T/100\pi}$.  Then \eqref{nolate} imples
	\[
	\PP(\Late_{\ell_0}[T]) \leq C_0 e^{-c_0 \ell_0}.
	\]
If $T$ is sufficiently large, say $T>T_0(\gamma)$, then the right side is less than $T^{-(\gamma+1)}$.

Now define
\[
m_0 = (1000/b)\ell_0.
\]
Lemma \ref{lem:earlyimplieslate} implies that if $m_0 \geq C_3 \log T$, then
\[
\PP(\Early_{m_0}[T]\cap \Late_{\ell_0}[T]^c)  \leq T^{-(\gamma+1)}.
\]
so that
\[
\PP(\Early_{m_0}[T])  < 2 T^{-(\gamma+1)}.
\]
For $k\geq 1$, define $\ell_k$ and $m_k$ recursively by
	\begin{align*} \ell_{k} &= \sqrt{(C_4 \log T)m_{k-1}} \\
	 m_k &= (1000/b) \ell_k.
	\end{align*}

By Lemma~\ref{lem:lateimpliesearly}, if $\ell_1 \geq C_4 \log T$, then
	\[ \PP (\Late_{\ell_1}[T]\cap \Early_{m_0}[T]^c ) \leq T^{-(\gamma+1)} \]
hence
	\[ \PP (\Late_{\ell_1}[T]) < 3 T^{-(\gamma+1)}. \]
	
By induction on~$k$, we find that
\[
\PP(\Late_{\ell_k}[T]) <  (2k+1) T^{-(\gamma+1)}
\]
and
\[
\PP(\Early_{m_k}[T])  < (2k+2) T^{-(\gamma+1)}
\]
provided $m_k \ge C_3\log T$ and $\ell_k \ge C_4\log T$.

The recursion gives
\[
\ell_{k}
= \alpha^{1 - 2^{-k}} \ell_0^{1/2^{k}}
\]
and
\[
m_k =  \alpha^{1 - 2^{-k}} m_0^{1/2^{k}}.
\]
where $\alpha = (1000/b) C_4\log T$.

Let  $C_5 = (1000/b) \max(C_3,C_4)$.  Then
\[
\ell_{k}, m_k \leq (C_5 \log T) T^{1/2^k}.
\]
For $k > (\log\log T)/\log 2$, we have
\[
T^{1/2^k} < 3.
\]
Therefore, if we choose $k$ largest so that $\ell_k \ge 3C_5 \log T$,
then
\[
k \le (\log\log T)/\log 2.
\]
Thus, in our final step we may use
\begin{align*}
\ell := \max(\ell_{k+1}, C_4 \log T) \le 3 C_5 \log T \\
m := \max(m_{k+1}, C_4 \log T) \le 3 C_5 \log T
\end{align*}
to obtain
\begin{align*}
\PP(\Late_{\ell}[T]) <  (2k + 3) T^{-(\gamma+1)} < \frac12 T^{-\gamma} \\
\PP(\Early_{m}[T]) <  (2k + 4) T^{-(\gamma+1)} < \frac12 T^{-\gamma}
\end{align*}
for all sufficiently large~$T$.

Let $r=\sqrt{T/\pi}$ and $a = 10C_5 \log r$.  Then $a > 3 C_5 \log T \geq \ell, m$.  By Lemma~\ref{lem:earlylatealternatedefinition}, we conclude that
	\begin{align*} \PP \left \{ \B_{r- a \log r} \subset A(\pi r^2) \subset \B_{r+ a\log r} \right \}^c
		&\leq \PP \left \{ \B_{r - \ell} \not\subset A(T) \right\} + \PP \left \{ A(T) \not \subset \B_{r+ m} \right \}  \\
		&\leq \PP (\Late_\ell[T]) +  \PP(\Early_m[T]) \\
		&< T^{-\gamma} < r^{-\gamma}.
	\end{align*}
This completes the proof of Theorem \ref{thm:logfluctuations}.

\section{Concluding Remarks} \label{s.conclusion}

We mention here a few open questions and possible extensions of our results, along with a brief overview of our forthcoming sequels.

\subsection*{Gaussian free field}

Let $\Lambda = [-n,n]^d \subset \mathbb Z^d$ and let $h$ be an instance of the discrete Gaussian free field on $\Lambda$ with zero boundary conditions.  That is \begin{enumerate}
\item write $$E(f) = \sum_{x,y \in \Lambda, |x-y|=1} |f(x)-f(y)|^2,$$
\item let $H$ be the $(2n-1)^d$-dimensional vector space of real-valued functions on $\Lambda$ that vanish on the boundary $\partial \Lambda$,
\item and choose $h$ from the measure $e^{-E(h)/2}dh$ (normalized to be a probability measure), where $dh$ is Lebesgue measure on $H$.
\end{enumerate}

\begin{table}
\begin{center}
\begin{tabular} {|l|l|l|l|}
\hline
 &  $d=1$ & $d=2$ & $d \geq 3$ \\[3pt]
\hline
Order of $h(x)$ at typical $x \in \Lambda$  & $\sqrt{n}$ & $ \sqrt{\log n}$ &  $1$ \\[3pt]
\hline
Order of $\max_{x \in \Lambda} |h(x)|$ & $\sqrt{n}$ & $\log n$ & $\sqrt{\log n}$\\[3pt]
\hline
Order of  $|\Lambda|^{-1} \sum_{x \in \Lambda} \phi(x/n) h(x)$  & $\sqrt{n}$ & $1$ & $n^{1-d/2}$ \\[3pt]
\hline
\end{tabular}
\end{center}
\caption{\label{table:gfforder} Order of various kinds of fluctuations for the discrete Gaussian free field in~$\Z^d$.  In the bottom row, $\phi: [-1,1]^d \to \R$ is a smooth test function.
}
\end{table}

Table~\ref{table:gfforder} gives the orders of various observables associated to~$h$.  (See \cite{BZ10} for recent results on the maximum of $h$ in two dimensions, and \cite{Sh} for a general survey on the GFF.)  Perhaps surprisingly, the orders in the last two rows do not change if we replace $\Lambda$ with a fixed co-dimension one subset of $\Lambda$.  This shows that weighted averages of $h$ are remarkably well concentrated (and become more concentrated in higher dimensions).
Appropriately normalized, they scale to Gaussian random variables with an asymptotic variance that depends $\phi$.

Next, if we let \[ L(z) = \sqrt{\frac{\min \{n | z \in A(n) \}}{\pi}} - |z| \] represent the ``lateness'' of the vertex $z$, then we expect $L$ and its higher dimensional analogues to behave similarly to $h$ in all dimensions.  (This is trivial for dimension $d=1$.) Our sequel paper, \emph{Internal DLA and the Gaussian free field} \cite{JLS11} addresses the bottom row (for certain test functions $\phi$) and shows that $L$ has a variant of the continuum Gaussian free field as a scaling limit.

\subsection*{Lower bounds on fluctuations}
The present paper established the upper bound of $\log n$ for the fluctuations of $L$ corresponding to the middle square of Table~\ref{table:gfforder}.  Simulations~\cite{MM,FL} indicate that fluctuations of this order are really present, and proving this is a very natural open problem.  One possible approach would be to use the methods of this paper to show that for small enough~$c$, the {\em expected} number of aberrations of size $c \log n$ (or $c\sqrt {\log n}$ in dimension $d \geq 3$) in $A(n)$ grows like a {\em positive} power of~$n$ --- and then using a second moment estimate to show that infinitely many such aberrations occur almost surely.  The main new ingredient required would be to control the correlations between fluctuations in two different directions.

\subsection*{Higher dimensions}
The $d\geq 3$ sequel to this paper \cite{JLS10} uses variants of the arguments presented here.  The function $H_\zeta$ is replaced by a discrete Green's function for a ball in $\Z^d$ (with zero boundary conditions).  The obvious analogue of $H_\zeta$ (constructed from differences of $\Z^d$ Green's functions) does not have spherical level surfaces.  This causes some extra complication in the estimates involved in Lemmas~\ref{Hsize} and~\ref{Hlevel}.  It becomes necessary to estimate the discrete Green's function on the ball directly, instead of building on known estimates for the $\Z^d$ Green's function.  The argument requires a few other modifications as well (for example, taking $\zeta$ to be a point inside the ball, instead of on the boundary).

\subsection*{Fluctuations in a fixed direction}
The bound of order $\log r$ in Theorem~\ref{thm:logfluctuations} measures the ``fluctuations in the worst direction.''  Another interesting quantity is the fluctuation in a specific direction, say along positive $x$-axis, as measured by the differences
	\[ \max \{x>0 \mid (x,0) \in A(\pi r^2)\} - r \]
and	
	\[ r - \min \{x>0 \mid (x,0) \notin A(\pi r^2) \}. \]
Based on the analogy with the discrete Gaussian free field, we believe that these random variables have standard deviation of order $\sqrt{\log r}$ in two dimensions, and that their higher dimensional analogues have $O(1)$ standard deviation.  These orders match those in the top line of Table~\ref{table:gfforder}.

In numerical experiments on a related quantity, Meakin and Deutch~\cite{MD} found order $\sqrt{\log r}$ fluctuations in dimension~$2$ and $O(1)$ fluctuations in dimension~$3$.

We have not been able to adapt our bootstrapping argument to prove bounds of this type on fluctuations in a fixed direction.  An $m$-early point where $m = \sqrt{\log r}$ contributes only $\sqrt{\log r}$ to the martingale $M_\zeta$, which is swamped by the logarithmic error term in the discrete mean value property, Lemma~\ref{Hlevel}(c).

\subsection*{Derandomized aggregation}
James Propp around 2001 proposed a growth process called \emph{rotor-router aggregation}, in which particles started at the origin in $\Z^2$ perform deterministic analogues of random walks until reaching an unoccupied site.  Large-scale simulations show that the aggregates produced in this way are substantially smoother than IDLA, with even smaller fluctuations from circularity~\cite{FL}. Indeed, part of Propp's motivation was to study how much of the random fluctuation can be removed by passing to the deterministic analogue.  The data indicate that the deviation from circularity in rotor-router aggregation may even be bounded independent of the size of the cluster; that is, $\B_{r-C} \subset A(\pi r^2) \subset \B_{r+C}$ for an absolute constant $C$.
Yet the current best upper bounds (order $\log r$ on the inside and $r^{1/2} \log r$ on the outside, \cite{LP09}) remain distant from this goal.  In trying to adapt our methods to rotor-router aggregation, we can define a deterministic analogue of the martingale $M_\zeta(t)$ by summing the function $H_\zeta$ over the rotor-router aggregate $A(t)$.  At this point an interesting question arises: what should take the place of the martingale representation theorem for this deterministic process?

\subsection*{External DLA}

External DLA, defined by Witten and Sander~\cite{WS}, produces intricate branching structures instead of a smooth spherical surface.  Typically it is defined with particles starting from infinity, but there is also a version in which particles start at the origin.  In this version, the initial occupied set $A(0)$ is the set of sites outside some extremely large origin-centered ball.  At time $n$ we start a random walk at the origin and stop when it first reaches a site $z_n$ {\em adjacent to} $A(n)$.  Then we set $A(n+1) = A(n) \cup \{z_n\}$.

Both this paper and the sequel (about the Gaussian free field) use a similar technique: namely, they both use geometric arguments to control the quadratic variation of continuous martingales obtained by summing the values of discrete harmonic functions over sets of particle locations.
One intriguing aspect of this work is that all of the martingales we utilize are still martingales for external DLA (grown from the origin).  Indeed, Kesten's upper bound on the growth rate of external DLA~\cite{Kesten} uses a closely related martingale.

Because the geometry of external DLA is much more intricate, however, it becomes more difficult to control the quadratic variation of the martingales.  In our argument, the interlinked inner and outer estimates --- indeed, even the notions of early and late points --- rely on the fact that the cluster fluctuates around a predictable asymptotic shape (the disk).  In external DLA, on the other hand, even the large-scale geometry is random.

\appendix
\section{Appendix: Proof of Lemmas~\ref{lem:thickening} and~\ref{lem:aprioriestimates}}

\subsection{Bounds on the probability of thin tentacles in all dimensions}
\label{sec:tentacle}

This appendix is devoted to proving the following lemma, which strengthens Lemma~\ref{lem:thickening}.  Since we anticipate using this result to bound the fluctuations of internal DLA in higher dimensions as well, we treat all dimensions $d \geq 2$.

\begin{lemmaA}
\label{lem:strongnotentacle}
In dimension $d=2$, there are positive absolute constants $b$, $C_0$, and $c$ such that
for all real numbers $m >0$ and all $z \in \Z^2$ with $0 \not \in \B(z,m)$,
\begin{equation} \label{e.d2tentaclebound}
\PP \left\{ z \in A(n),~ \# (A(n) \cap \B(z,m)) \leq b m^2 \right\} \leq
C_0 e^{-c m^2/\log m}.
\end{equation}

For each dimension $d \geq 3$, there are positive absolute constants $b$, $C_0$, and $c$ (depending only on $d$) such that
for all real numbers $m >0$ and all $z \in \Z^2$ with $0 \not \in \B(z,m)$,
\begin{equation} \label{e.d3tentaclebound}
\PP \left\{ z \in A(n),~ \# (A(n) \cap \B(z,m)) \leq b m^d \right\} \leq
C_0 e^{-c m^2}.
\end{equation}
\end{lemmaA}

Note that even the thinnest possible tentacle, a path of length~$m$, can be formed using only $m(m+1)/2$ random walk steps: for $j=1,\ldots,m$, the $j$-th particle to enter $\B(z,m)$ takes $j$ steps along the path and occupies the next site on the path.  This shows that the exponent~$m^2$ on the right side of \eqref{e.d3tentaclebound} is best possible.

Instead of working with the ball $\B(z,m)$ it will be slightly more convenient to use the box $\tilde \B(z,m) := z + [-m,m]^d \subset \mathbb Z^d$.  Since the volume of a ball is (up to a $d$ dependent constant factor) the same as that of a box of the same radius, Lemma~A holds as stated if and only if it holds with $\B(z,m)$ replaced by $\tilde \B(z,m)$.

Fix $z \in \Z^d$, and for $0 \leq j \leq m$ let $S_j$ be the set of points $x = (x_1, \ldots, x_d) \in \Z^d$ such that
	\[ \max_{i=1}^d  |x_i-z_i| = m-j. \]  The sets $S_j$ partition $\tilde \B(z,m)$ into $m+1$ concentric shells, indexed from the outside inward.  Since $\Z^2 \backslash S_j$ is disconnected, at least one site in shell $S_j$ must join the cluster $A(n)$ before any site in $S_{j-1}$ can join.

It is enough to show that \eqref{e.d2tentaclebound} and \eqref{e.d3tentaclebound} hold if we replace $A(n)$ with the set~$A$, where $A = A(n_z) \cap \tilde \B(z,m)$, and $n_z$ is defined to be the smallest $n$ for which $z \in A(n)$.  Now write $a_j := \#(A \cap S_j)$ for $j \in \{0,1,\ldots,m \}$.  The cluster $A(n_z)$ intersects each shell $S_j$, so $a_j \geq 1$ for all $j=0,1,\ldots,m$.

Given $m$, the number of possibilities for the sequence $a_j$ is at most $((2m+1)^d)^m$, which is much smaller than the reciprocal of  the right side of \eqref{e.d2tentaclebound} and \eqref{e.d3tentaclebound}.  It thus suffices enough to show that, for each {\em fixed} sequence $\aa = (a_j)_{j=0}^m$ with $a_j \geq 1$ and $\sum_{j=0}^m a_j \leq bm^d$, \eqref{e.d2tentaclebound} and \eqref{e.d3tentaclebound} hold when the left side is replaced by the probability of the event
	\[ \mathcal{T}_\aa := \bigcap_{j=0}^m \left\{\#(A \cap S_j) = a_j \right\}. \]
We emphasize that $C_0$ and $c$ are still absolute constants; in particular, they do not depend on the sequence~$\aa$ or on~$m$.  We will henceforth assume that such a sequence~$\aa$ has been fixed, and bound the probability $\PP(\mathcal{T}_\aa)$.

\begin {figure}[htbp]
\begin {center}
\includegraphics [width=5.2in]{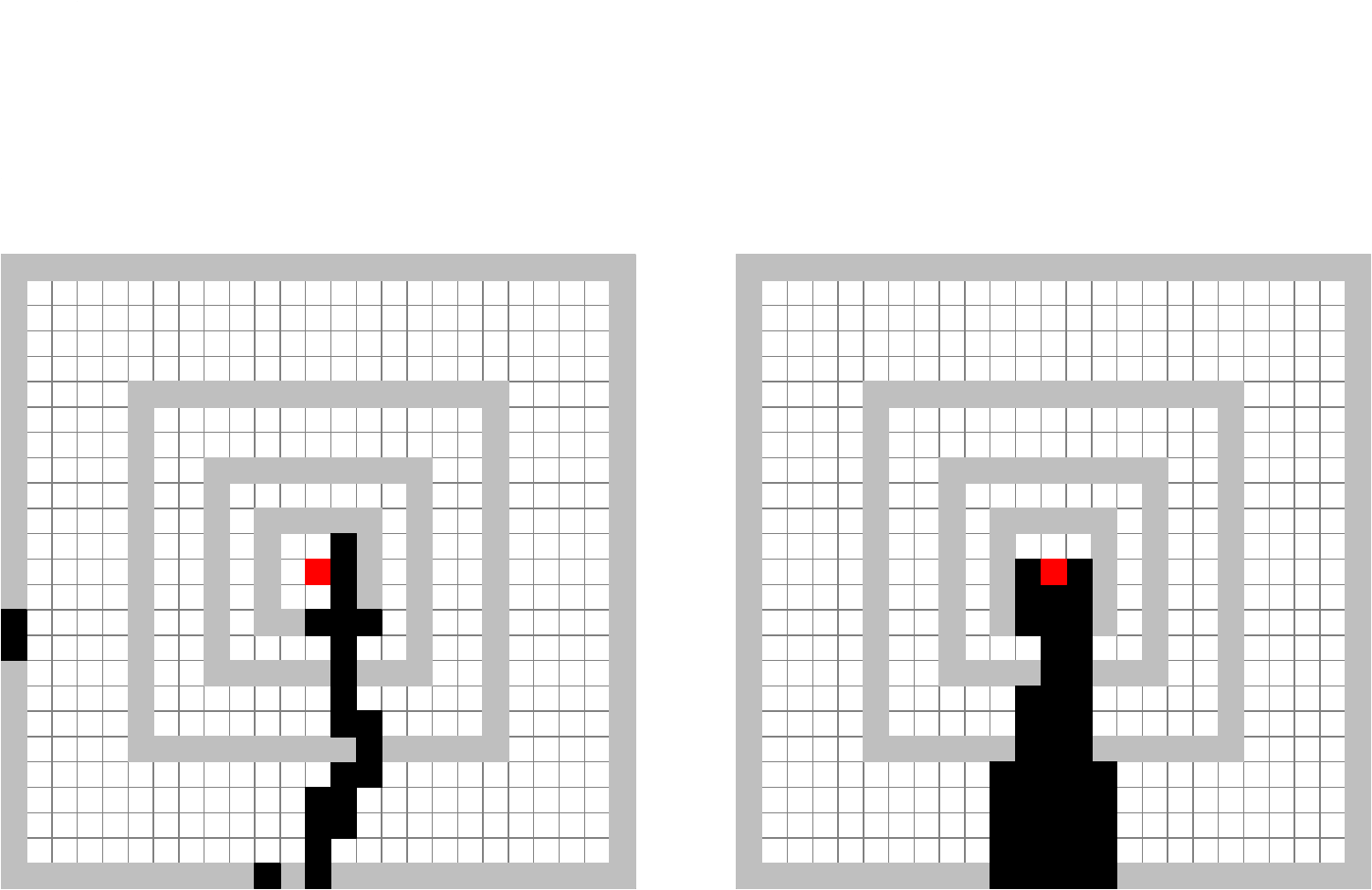}
\caption {\label{shells} On the left: a possible value for $A$ (black, with $z$ in red at the center).  Here $m=12$ and the sequence $a_0, a_1, \ldots, a_{12}$ (giving number of black squares in the square shell $S_j$, indexed from outside inward) is given by $4, 1, 2, 2, 2, 1, 2, 1, 1, 1, 3, 3, 1$.  Taking $c'=1/2$, the corresponding sequence $b_0, \ldots, b_{12}$ is given by $5, 5, 5, 5, 5, 3, 3, 3, 2, 2, 3, 3, 3$, and a corresponding tower of squares is illustrated on the right.  The sizes of the squares are $\beta_1 = 5$, $\beta_2 = 3$, $\beta_3 = 2$, $\beta_4 = 3$.  The grey shells indicate the $S_{\alpha_i}$, where $\alpha_0 = 0$, $\alpha_1 = 5$, $\alpha_2 = 8$, and $\alpha_3 = 10$.  Since $c'=1/2$, the number of black squares in the ``annulus'' formed by the union of the $S_j$ crossing a given square on the right (other than the innermost square) is always at least twice the number black of squares in the corresponding annulus on the left.
}
\end {center}
\end {figure}

Now we will construct a related sequence $b_j$ indexed by $[0,m]$ with two nice properties:

\begin{enumerate}
\item $b_j$ is roughly the kind of sequence $a_j$ that one would see if $A$ were a tower of cubes.  More precisely, let $\beta_1,\ldots,\beta_k$ be positive integers summing to~$m$, representing the side lengths of the cubes.  For $i \in \{0,\ldots,k\}$, let $\alpha_i = \beta_1 + \cdots + \beta_i$.  For $j \in \{0,\ldots,m\}$, let~$\gamma_j$ be the unique integer $i$ such that $\alpha_{i-1} \leq j < \alpha_{i}$.  Then set
	\[ b_j = \left(\beta_{\gamma_j}\right)^{d-1}. \]
\old{	
$b_j$ is equal to $\beta_1^{d-1}$ for $j \in \{0,1,\ldots, \beta_1-1\}$, then equal to a constant $\beta_2^{d-1}$ for $j \in \{\beta_1, \ldots \beta_1 + \beta_2 -1 \}$, and so forth.  Some notation:  for $j \geq 0$ let $\alpha_j = \sum_{i < j} \beta_j$.  The $\alpha_j$ correspond to the starting points of the cube intervals, except that if $k$ is the total number of cube intervals then $\alpha_k = m$.
}
Note that $\sum_{j=0}^m b_j = \sum_{i=1}^k \beta_i^d$.

\item There is a small constant $c'$ such that for each $i=1,\ldots,k-1$ (but possibly not for $i=k$)
	\[ c' (\beta_i/2)^d \leq \sum_{j=\alpha_{i-1}}^{\alpha_{i}-1} a_j \leq c' (\beta_i)^d. \]
\old{
In each interval corresponding to a cube (except possibly the last such interval), the sum of $a_j$ over the interval is at most a fixed small constant $c'$ times the sum of $b_j$ over that interval (i.e., $c'$ times the volume of the cube) and at least $c''=c'/2^d$ times the sum of $b_j$ over that interval.
}
\end{enumerate}

Once we have fixed $c'$, we can construct the sequence $b_j$ from $a_j$ in a deterministic way by induction:  Given the cube side lengths $\beta_1,\ldots,\beta_{i-1}$, take $\beta_{i}$ to be the smallest possible such that the second condition holds.

Next we make explicit our choice of $c'$.  For $x \in \Z^d$, write $\PP_x$ for the probability measure associated to simple random walk $(X(t))_{t=0}^{\infty}$ in $\Z^d$ with $X(0)=x$.  By the local central limit theorem~\cite[\textsection 2]{LL10}, if~$t$ is even, then for each site $y$ of even parity in the ball $\B(x,\sqrt{t})$, we have
	\[ \PP_x (X(t) = y) \geq c_0 t^{-d/2} \]
for a constant $c_0>0$ depending only on~$d$.  Let $c_1 = \omega_d 2^{-d-5} c_0$, where $\omega_d$ is the volume of the unit ball in $\R^d$.

By \cite[Prop.\ 2.1.2]{LL10}, there is a constant $c_2>0$ depending only on $d$, so that
	\begin{equation} \label{e.youcannotescape}
	\PP_0 \left \{ \max_{0 \leq s \leq t} |X(s)| \geq \sqrt{t}/c_2 \right \} < c_1.
	\end{equation}
Now take $c' = \omega_d 2^{-d-4} c_2^d$ in the above construction of the sequence $b_j$ from $a_j$.  This defines a sequence of shells $S_{\alpha_i}$ as in Figure~\ref{shells}.  According to the next lemma, if $A'$ is any subset of $\Z^d$ satisfying $A' \cap S_j \leq a_j$ for $j=0,\ldots,m$, then simple random walk started from any point in the $i$-th shell $S_{\alpha_i}$ has at least a constant $c_1>0$ probability of exiting $A'$ before reaching $S_{\alpha_{i-1}} \cup S_{\alpha_{i+1}}$.

\begin{lemma} \label{lem:trialsuccessrate}
Let $A'$ be any set for which $\# (A' \cap S_j) \leq a_j$ for each $j=0,\ldots,m$.  Fix $1 < i < k$ and $x \in S_{\alpha_i}$.  For simple random walk in $\Z^d$ started at $x$, let $T$ be the first hitting time $S_{\alpha_{i-1}} \cup S_{\alpha_{i+1}}$, and let $T'$ be the first hitting time of $(A')^c$.  Then
	\[ \PP_x \left\{T > T' \right\} \geq c_1. \]
\end{lemma}

\begin{proof}
Let $h = \min(\beta_i, \beta_{i+1})$, and let~$t$ be the greatest even integer $\leq (c_2 h)^2$.
Let $r=c_2 h$.  By the local central limit theorem, for each site $y$ of even parity in the ball $\B(x,r)$, we have
	\[ \PP_x (X(t) = y) \geq c_0 r^{-d}. \]

Suppose first that $\beta_{i+1} \leq \beta_{i}$.  Let $B$ be the set of points of even parity in the intersection
	\[  \B(x,r) \cap \bigcup_{j=\alpha_i}^{\alpha_{i+1}-1} S_j. \]
Note that $\# B \geq \# \B(x,r) / 2^{d+2} \geq (\omega_d 2^{-d-3}) r^d$.
Hence by our choice of $c'$,
	\[ \# (A' \cap B) \leq  \sum_{j=\alpha_{i}}^{\alpha_{i+1}-1} a_j
	\leq c' \beta_{i+1}^d  = c' h^d < \#B / 2. \]
Hence
	\[ \PP_x \{ X(t) \notin A' \} \geq \sum_{y \in B \cap (A')^c} \PP_x(X(t)=y) \geq (\#B / 2) \frac{c_0}{r^d} > 2c_1. \]
By \eqref{e.youcannotescape}, we have
	\[ \PP_x (T \leq t) < c_1. \]
We conclude that
	\begin{align*} \PP \{ T > T' \}
	 &\geq \PP_x \{T > t, ~X(t) \notin A' \} \\
	 &\geq 1 - \PP_x \{T \leq t \} - \PP\{ X(t) \in A' \} \\
	 &> 1 - c_1 - (1-2c_1).
	\end{align*}
	
The proof in the case $\beta_{i+1} > \beta_i$ is identical, using the shells $\bigcup_{j=\alpha_{i-1}+1}^{\alpha_i} S_j$ in the definition of the set $B$.
\end{proof}

During the formation of the IDLA cluster $A$, each time a particle hits a shell $S_{\alpha_i}$ ($1<i<k$) for the first time, we begin a ``trial'' which ends either when the particle stops walking (due to exiting the cluster and occupying a previously unoccupied site) or when it next hits $S_{\alpha_{i-1}} \cup S_{\alpha_{i+1}}$.  In the former case, we call the trial a ``failure.''  Otherwise, we call it a ``success.''

Let $F_s$ be the event that the $s$-th trial fails, and let $\mathcal{G}_{s}$ be the $\sigma$-field generated by all random walks generating the IDLA cluster up until the start of the $s$-th trial.  Let $A_s$ be the occupied cluster at the time when the $s$-th trial starts, and let $E_s= \bigcap_{j=0}^m \{ \# (A_s \cap S_j) \leq a_j \}$.
By Lemma~\ref{lem:trialsuccessrate}, we have
	\[ \PP(F_s | \mathcal{G}_{s} ) \one_{E_s} \geq c_1 \one_{E_s}. \]
In other words, on the event $E_s$, once the $s$-th trial begins (and we condition on the past of the process --- in particular on all previous trials) the trial fails with probability at least probability $c_1$.
\old{
will call the subsequent behavior of the particle (until the particle either stops moving or hits a distinct shell $S_{\alpha_{i'}}$) a ``trial.''  If the particle stops in $\tilde \B(z,m)$ before it reaches either $S_{\alpha_{i-1}}$ or $S_{\alpha_{i+1}}$ (and is deposited), we call the trial a failure.  Otherwise we call it a success.
 On the event that $A$ corresponds to the prescribed sequence $a_j$, there are at most $\sum_{j=0}^m a_j$ failures (since each failure adds a new site to~$A$).
 If $\# A$ is a lot less than $c_1$ times the number of trials, then we would expect that having so few failures would be unlikely.
 }

Since each failure adds a new site to~$A$, there are at most $\# A$ failures.
Next we estimate the total number of trials.  Each particle that makes it to $S_{\alpha_i}$ has to first pass through all the $S_{\alpha_{i'}}$ for $i' < i$.  Thus the process of successful growing a set $A$, such that $\#(A \cap S_j) = a_j$ for all $j$, necessarily includes at least $T = \sum a_j \gamma_j$ trials.  Hence on the event $\mathcal{T}_\aa$, there are at least $T$ trials, at most $\sum a_j$ of which fail; and conditional on all past trials, each trial fails with probability at least $c_1$.

Before we put an explicit bound on $T$, we'd like to find a way to reduce the problem to the case that
\begin{equation} \label{e.ATnice} \# A = \sum a_j < c_1 T/2 \end{equation} (i.e., failure is required to occur at less than half the expected rate).  In other words, we'd like to reduce the case that $\sum a_j \gamma_j$ is larger (by a fixed constant factor $2/c_1$) than $\sum a_j$.  This will certainly be the case, for example, if more than a fraction $1/4$ of the sum $\sum a_j$ comes from values of $j$ for which $\gamma_j > 8/c_1$.

Let $R$ be a fixed integer greater than $8/c_1$.  Let $I$ be the smallest positive integer such that
	\[ \sum_{j = \alpha_{IR}}^m a_j < 4 \sum_{j = \alpha_{IR}+1}^m a_j. \]
Then consider the smaller box we get by replacing~$m$ by $m - \alpha_{IR}$.  We can take~$b$ small enough so that~$\beta_i^d$ is at most an arbitrarily small constant times~$m^d$, and hence~$\beta_i$ is at most some arbitrarily small constant times $m$.  Indeed, since~$R$ is just a constant, we can take~$b$ small enough so that $\sum_{i=1}^R \beta_i$ is at most an arbitrarily small constant times~$m$.  Since the sequence
	\[ \sum_{j = \alpha_{iR}}^{\alpha_{(i+1)R} - 1} a_j \]
decreases by a factor of at least~$4$ each time~$i$ increases by~$1$ (up until~$i = I$), we also obtain that the sequence
	\[ \beta_{iR} + \cdots + \beta_{(i+1)R -1} \]
is also bounded above by a sequence that decreases exponentially in~$i$, and hence (by making~$b$ small enough) we can arrange so that $\sum_{i=1}^{IR} \beta_i$ is less than $m/2$.  Then if we replace $m$ with $m - \alpha_{IR}$ (and replace $A$ by its intersection with the corresponding smaller box), we reduce the problem to one for which \eqref{e.ATnice} holds and $m$ has decreased by at most a factor of $2$.

Now assuming \eqref{e.ATnice}, the usual large deviations bound (Cramer's theorem) implies that the probability of seeing $T$ trials with only $c_1 T/2$ or fewer failures is exponentially unlikely in $T$.  Note that~$T = \sum_{j=0}^m a_j \gamma_j$ is up to a multiplicative constant (between $c'/2^d$ and $c'$) the same as $\sum_{j=0}^m b_j \gamma_j =  \sum_{i=1}^k i \beta_i^d$.   Hence the proof of Lemma \ref{lem:strongnotentacle} is completed by the following lower bound on $\sum i \beta_i^d$.

\begin{lemma} \label{lem:trialnumbound}
Fix $d \geq 2$ and $m \geq 1$, and consider an integer $k$ and sequence of real numbers $\beta_1, \beta_2, \ldots, \beta_k$ for which the following hold:
\begin{enumerate}
\item $\beta_i \geq 1$ for all $i \in \{1,\ldots, k \}$.
\item $\sum_{i=1}^k \beta_k = m+1$.
\item $E(\beta) := \sum_{i=1}^k i \beta_i^d$ is minimal among sequences satisfying the two conditions above.
\end{enumerate}
There is a constant $c(d) > 0$, depending only on $d$, such that \begin{equation} \label{e.betacases} \begin{cases} E(\beta) \geq c(d) m^2/\log m, & d=2; \\ E(\beta) \geq c(d) m^2, & d \geq 3. \end{cases} \end{equation}
\end{lemma}
\begin{proof}
The first two conditions imply that $k \leq m+1$.  By Holder's inequality,
	\[ m+1 = \sum_{i=1}^k \beta_i \leq \left( \sum_{i=1}^k i \beta_i^d \right)^{1/d} \left( \sum_{i=1}^k i^{-d'/d} \right)^{1/d'} \]
where $d'$ is such that $1/d + 1/d' = 1$.  Hence
	\[ E(\beta) \geq (m+1)^d \left( \sum_{i=1}^k i^{-d'/d} \right)^{-d/d'}. \]
If $d>2$, then $d'/d < 1$ and the right side is at least a constant times $m^d k^{(1-\frac{d'}{d})(-\frac{d}{d'})} = m^d k^{2-d} \geq m^2$.  In the case $d=2$, we have $d'/d=1$ and the right side is at least a constant times $m^2 (\log k)^{-1} \geq m^2 (\log m)^{-1}$.
\end{proof}

\subsection{Bound on the probability of very late points}
\label{sec:aprioriproof}

In this section we prove~Lemma \ref{lem:aprioriestimates}.  Lawler, Bramson and Griffeath give an exponential bound on the probability of $\ell$-late points in the case that $\ell$ is comparable to the radius of the cluster.  Specifically, they show (see Lemma~6 of~\cite{LBG} and the displayed equation immediately following it) that for any $\epsilon>0$ there is a constant~$c_2$ depending on $\epsilon$ such that for all $z \in \B_{(1-\epsilon)n}$
	\begin{equation} \label{LBGbound1} \PP( z\notin A((1+\epsilon)\pi n^2 )) < 2 e^{-c_2 n}. \end{equation}
Fix $r>0$.  Taking $\epsilon' = \epsilon/2$ and $n=(1-\epsilon')r$ in \eqref{LBGbound1}, we have $(1-\epsilon) r < (1-\epsilon')n$ and $\pi r^2 > (1+\epsilon') \pi n^2$, so for all $z \in \B_{(1-\epsilon )r}$,
	\begin{equation} \label{LBGbound2} \PP( z\notin A(\pi r^2 )) < 2 e^{-c_2 n} \leq 2 e^{-cr} \end{equation}	
where $c$ is a constant depending on $\epsilon$.
\old{
Summing \eqref{LBGbound2} yields
	\begin{align} \PP( \B_{(1-\epsilon)r} \not\subset A(\pi r^2))
	&\leq \sum_{z \in \B_{(1-\epsilon)r}} \PP(z \notin A(\pi r^2))\nonumber \\
	&\leq 2 \pi r^2 e^{-cr} \nonumber \\
	&\leq C'_0 e^{-c'r} \label{LBGbound3}
	\end{align}
for constants $c', C'_0 >0$ depending on $\epsilon$.
}

By definition,
	\[  \Late_\ell[100\pi \ell^2] = \bigcup_{z\in \B_{9\ell}} \left\{ z \notin A(\pi(|z|+\ell)^2) \right\}. \]
Taking $r = |z|+\ell$ in \eqref{LBGbound2} with $\epsilon=\frac{1}{10}$, we have for all $z \in \B_{9\ell}$
 	\[ \PP \{ z \notin A(\pi r^2) \} < 2e^{-cr} \leq 2e^{-c\ell}. \]

Summing over $z \in \B_{9\ell}$ we obtain
	\begin{align*} \PP (\Late_\ell[100\pi \ell^2])
	\leq 200 \pi \ell^2 e^{-c\ell}.
	\end{align*}
The right side is at most $C_0 e^{-c_0 \ell}$ for suitable constants $C_0,c_0>0$, which proves \eqref{nolate}.

\end{document}